\documentclass[11pt]{amsart}
\usepackage{amsmath, amssymb, latexsym}
\usepackage{mathrsfs}
\usepackage{enumerate}
\usepackage[all, knot]{xy}
\usepackage{color}

\newtheorem{thm}{Theorem}[section]
\newtheorem{prop}[thm]{Proposition}
\newtheorem{cor}[thm]{Corollary}
\newtheorem{lemma}[thm]{Lemma}

\theoremstyle{definition}
\newtheorem{defn}[thm]{Definition}

\theoremstyle{remark}

\newtheorem{THEOREM}{Theorem}

\newenvironment{red}
{\relax\color{red}}
{\hspace*{.5ex}\relax}

\newcommand{\ber}{\begin{red}}
\newcommand{\er}{\end{red}}

\newenvironment{verd}
{\relax\color{magenta}}
{\hspace*{.5ex}\relax}

\newcommand{\bg}{\begin{verd}}
\newcommand{\eg}{\end{verd}}

\numberwithin{equation}{subsection}

\newcommand{\Z}{\mathbb{Z}}
\newcommand{\Q}{\mathbb{Q}}
\newcommand{\C}{\mathbb{C}}
\newcommand{\A}{\mathbb{A}}

\newcommand{\g}{\mathfrak{g}}

\newcommand{\Hom}{\mathrm{Hom}}
\newcommand{\End}{\mathrm{End}}

\newcommand{\wt}{{\rm wt}}

\newcommand{\Span}{{\rm Span}}

\newcommand{\Top}{{\rm Top}}
\newcommand{\Soc}{{\rm Soc}}
\newcommand{\Rad}{{\rm Rad}}

\newcommand{\proj}{\mathrm{proj}}
\newcommand{\im}{\mathrm{Im}}
\newcommand{\Ker}{\mathrm{Ker}}

\newcommand{\rlQ}{\mathsf{Q}}   
\newcommand{\wlP}{\mathsf{P}}   
\newcommand{\weyl}{\mathsf{W}}  
\newcommand{\cmA}{\mathsf{A}}  

\newcommand{\ST}{\mathsf{ST}}   
\newcommand{\res}{\mathrm{res}}   
\newcommand{\sg}{\mathfrak{S}}   
\newcommand{\F}{\mathcal{F}}   

\newcommand{\bR}{\mathbf{k}}   
\newcommand{\sB}{\mathbb{B}}   
\newcommand{\KLR}{R^{\Lambda_0}}   

\newcommand{\sBox}[1]
{
\xy
(-2,-2)*{};(2,-2)*{} **\dir{-};
(-2,2)*{};(2,2)*{} **\dir{-};
(-2,-2)*{};(-2,2)*{} **\dir{-};
(2,-2)*{};(2,2)*{} **\dir{-};
(0,0)*{#1};
\endxy
}

\newcommand{\ssBox}[1]
{
\xy
(-1.5,-1.5)*{};(1.5,-1.5)*{} **\dir{-};
(-1.5,1.5)*{};(1.5,1.5)*{} **\dir{-};
(-1.5,-1.5)*{};(-1.5,1.5)*{} **\dir{-};
(1.5,-1.5)*{};(1.5,1.5)*{} **\dir{-};
(0,0)*{ _{#1}};
\endxy
}

\begin{document}

\title[Representation type of finite quiver Hecke algebras of type $A^{(2)}_{2\ell}$]
{Representation type of finite quiver \\
Hecke algebras of type $A^{(2)}_{2\ell}$}

\author[Susumu Ariki]{Susumu Ariki$^{1}$}
\thanks{$^1$ This work is supported in part by JSPS, Grant-in-Aid for Scientific Research (B) 23340006.}
\address{Department of Pure and Applied Mathematics, Graduate School of Information
Science and Technology, Osaka University, Toyonaka, Osaka 560-0043, Japan}
\email{ariki@ist.osaka-u.ac.jp}

\author[Euiyong Park]{Euiyong Park$^{2}$}
\thanks{$^2$ This work is supported
by JSPS Postdoctoral Fellowship for Foreign Researchers.}
\address{Department of Mathematics, University of Seoul, Seoul 130-743, Korea}
\email{epark@uos.ac.kr}




\begin{abstract}
We study cyclotomic quiver Hecke algebras $R^{\Lambda_0}(\beta)$ in type
$A^{(2)}_{2\ell}$, where $\Lambda_0$ is the fundamental weight. The algebras
are natural $A^{(2)}_{2\ell}$-type analogue of Iwahori-Hecke algebras
associated with the symmetric group, from the viewpoint of the Fock space
theory developed by the first author and his collaborators. We give a
formula for the dimension of the algebra, and a simple criterion to
tell the representation type. The
criterion is a natural generalization of Erdmann and Nakano's for the
Iwahori-Hecke algebras. Except for the examples coming from cyclotomic Hecke
algebras, no results of these kind existed for cyclotomic quiver Hecke
algebras, and our results are the first instances beyond the case of
cyclotomic Hecke algebras.
\end{abstract}

\maketitle


\vskip 2em

\section*{Introduction}

Block algebras of the finite Hecke algebra associated with the symmetric group have been studied extensively in 1990's. Let $\bR$ be an algebraically closed field,
$q\in\bR^\times$ the parameter of the Hecke algebra, and $e$ the \emph{quantum characteristic} defined by
$$
e=\min\{ a\in\Z_+ \mid 1+q+\cdots+q^{a-1}=0 \;\text{holds in $\bR$.} \}.
$$
Then, the block algebras are parametrized by pairs of an $e$-core partition and an $e$-weight.
This well-known fact, the \emph{Nakayama conjecture} for Hecke algebras, was proved by James and Mathas \cite[Thm. 4.29]{JM97} based on earlier work by Dipper and James \cite{DJ87}. Nowadays it is understood as one of the results which fits in Fock space theory for cyclotomic Hecke algebras: the block algebras categorify weight spaces of the basic $\g(A^{(1)}_{e-1})$-module $V(\Lambda_0)$. In fact, this example was the origin of the conjecture by Lascoux, Leclerc and Thibon, where the Fock space theory started. Let $\delta$ be the null root and let $W$ be the Weyl group i.e. the affine symmetric group. We recall that every weight that appear in $V(\Lambda_0)$ has the form $w\Lambda_0-k\delta$, where $w\in W$ and $k\in\Z_{\geq0}$. Then,  in the Misra-Miwa realization of the Fock space, $w\Lambda_0$ is the $e$-core partition, and $k$ is the $e$-weight.
Thus, the modular representation theory of finite Hecke algebras turned out to be Lie theoretic, and we have various interesting questions to ask under the philosophy of categorification.
For example, research in the early stage focused on higher level analogues, that is, the modular representation theory of cyclotomic Hecke algebras, and finite Hecke algebras of type $BC$ (and type $D$ by Clifford theory) are special cases of them. See \cite{Ar01} and \cite{Ar07}, for example.

An important turning point for the philosophy was the introduction of  cyclotomic quiver Hecke algebras by Khovanov and Lauda in \cite{KL09}, \cite{KL11}.
The algebras are associated with dominant integral weight $\Lambda$, certain set of polynomials and non-negative integral linear combination $\beta$ of simple roots. We denote the algebra by $R^\Lambda(\beta)$.
As Rouquier \cite{R08} also developed another theory to categorify integrable highest weight modules, which generalizes his work with Chuang \cite{CR08}, it is also called the \emph{cyclotomic KLR algebra} by various authors. Further, theorems from Brundan and Kleshchev \cite{BK09}, Rouquier \cite{R08} and Lyle and Mathas \cite{LM07} combined imply that block algebras of the cyclotomic Hecke algebras are cyclotomic quiver Hecke algebras of type $A^{(1)}_{e-1}$. More precisely, we label the vertices of the Dynkin diagram with $\Z/e\Z$ as usual, and we let $A=(a_{ij})_{i,j\in\Z/e\Z}$ be the Cartan matrix. Then their quiver Hecke algebra is associated with polynomials
$$Q_{i,j}(u,v)=1, \text{if $a_{ij}=0$, and}\;\;
Q_{i,i+1}(u,v)=Q_{i+1,i}(v,u)=-(u-v)^{-a_{i,i+1}}, \text{for $i\in\Z/e\Z$.}$$
Therefore, the cyclotomic quiver Hecke algebras are indeed generalization of cyclotomic Hecke algebras.

Many papers on the quiver Hecke algebras have already appeared, e.g.  \cite{BKOP}, \cite{HMM09}, \cite{KP11}, \cite{KR08}, \cite{KR11}, \cite{LV11} or \cite{VV11}. But most of them study labeling or construction of irreducible modules. The only case which receives rather detailed study is type $A^{(1)}_{e-1}$, where the isomorphism theorem by Brundan-Kleshchev \cite{BK09} and Rouquier \cite{R08} allows us to reduce problems to well-studied cyclotomic Hecke algebras.

Our aim is to show some new results which go beyond this stage. Note that, beyond classifying or constructing irreducible modules,
first questions to be asked are representation type, shape of the Auslander-Reiten quiver and others.
In this paper, we start with $\KLR(\beta)$ in type $A^{(2)}_{2\ell}$ and their representation type.
As we have explained in the above, the algebras are natural $A^{(2)}_{2\ell}$-type analogues of the Iwahori-Hecke algebra associated with the symmetric group. Thus, we call them \emph{finite quiver Hecke algebras of  type $A^{(2)}_{2\ell}$}. We give a dimension formula for $\KLR(\beta)$, some structure theorem, and prove Erdmann-Nakano type theorem which tells the representation type of  $\KLR(\beta)$. Note that the Erdmann-Nakano theorem treats $e=2$ case separately, and tame representation type appears only in this case.
 The reader might expect that the statement in type  $A^{(2)}_2$ would be different from the other $A^{(2)}_{2\ell}$, for $\ell\geq2$, by the comparison with the $A^{(1)}_{e-1}$ cases.
However, we find that it is not the case, and the criterion is uniform for all $\ell$. Apart this point, the criterion is exactly the same as that of the Erdmann-Nakano theorem, for $e\geq3$, as follows.

\begin{THEOREM}[Theorem \ref{Thm: repn type of R(beta)}]
Suppose that $\KLR(\beta)$ is a finite quiver Hecke algebra of type $A^{(2)}_{2\ell}$, for $\ell\ge1$. We denote the Weyl group of type $A^{(2)}_{2\ell}$ by $\weyl$, $\weyl\Lambda_0$ the $\weyl$-orbit through $\Lambda_0$. There exist unique
$\kappa\in \weyl\Lambda_0$ and unique $k\in\Z_{\ge0}$ such that
$\beta=\Lambda_0-\kappa+k\delta$. Then, $\KLR(\beta)$ is
\begin{itemize}
\item[(1)]
simple if $k=0$,
\item[(2)]
of finite representation type but not semisimple if $k=1$,
\item[(3)]
of wild representation type if $k\ge2$,
\item[(4)]
and tame representation type does not occur.
\end{itemize}
\end{THEOREM}

The following is our dimension formula. We do not explain the notation here. Note that almost no result is known for the dimension of cyclotomic quiver Hecke algebras, except for block algebras of cyclotomic Hecke algebras. Another importance of Theorem B is that it tells when $e(\nu)=0$, for $\nu\in I^n$. It is repeatedly used in the proof of Theorem A.

\begin{THEOREM}[Theorem \ref{Thm: dimension formula}]
Let $\lambda \vdash n$ be a shifted Young diagram consisting of $n$ boxes. For $\beta  \in \rlQ^+$ with $|\beta|=n$ and $\nu, \nu' \in I^\beta$, we have
\begin{align*}
\dim e(\nu') \KLR(n) e(\nu) &= \sum_{\lambda \vdash n} 2^{-\langle d, \wt(\lambda)\rangle - l(\lambda)} K(\lambda, \nu')K(\lambda, \nu), \\
\dim  \KLR(\beta)  &= \sum_{\lambda \vdash n,\ \wt(\lambda)=\Lambda_0 - \beta } 2^{-\langle d, \wt(\lambda)\rangle - l(\lambda)} |\ST(\lambda)|^2, \\[5pt]
\dim  \KLR(n)  &= \sum_{\lambda \vdash n} 2^{-\langle d, \wt(\lambda)\rangle  - l(\lambda)} |\ST(\lambda)|^2.
\end{align*}
\end{THEOREM}

In proving our results, solution to the cyclotomic categorification conjecture from \cite{KL09} by Kang and Kashiwara in \cite{KK11} and \cite{Kash11} is used in an essential way. The strategy of the proof works for {\it general} affine types.
In particular, the argument in our paper gives a completely new proof of the original Erdmann-Nakano theorem \cite[1.2]{EN02} as well.
We will consider the other affine types than $A^{(1)}_{e-1}$ and $A^{(2)}_{2\ell}$ in separate papers.
We remark that Hecke-Clifford superalgebras and quiver Hecke superalgebras also categorify the basic representation $V(\Lambda_0)$ in type $A_{2\ell}^{(2)}$ and its quantization $V_q(\Lambda_0)$ respectively \cite{BK01, KKO12, KKT11}.
We think that the finite quiver Hecke algebras deserve detailed study because they give uniform generalization of the Iwahori-Hecke algebra. In the same spirit, studying uniform generalization of the Hecke-Clifford superalgebra by categorification using superalgebras might be a different direction to be pursued.

The structure of the paper is as follows. Section $1$ is for preliminaries. In section $2$, we recall cyclotomic quiver Hecke algebras and categorification results. In section $3$, we prove the dimension formula. In section $4$, we prove the Erdmann-Nakano type theorem for the representation type of $\KLR(\beta)$. Section $5$ is an appendix for generalized cellularity. We propose a generalized cellular algebra structure for finite quiver Hecke algebras of type $A^{(2)}_{2\ell}$. We are grateful to Mr. Mori for letting us know his work on generalized cellularity and some comments on our paper.

\vskip 1em

\section{Preliminaries}

In this section, we quickly review the Fock space of neutral fermions as a module over the Kac-Moody algebra of type $A_{2\ell}^{(2)}$.
The Fock space is realized in terms of combinatorics of shifted Young diagrams, and it will be used in a crucial manner
in proving the dimension formula for finite quiver Hecke algebras.

\subsection{Shifted Young diagrams} \label{Section: Young}

An array $\lambda$ of a finite number of boxes arranged in $l$ rows is a {\it shifted Young diagram of depth} $l$, if
\begin{itemize}
\item[(i)] the $i$th row from the top, for $1\le i\le l$, starts with its leftmost box in the $i$th column and there is no gap before it ends on the rightmost box,
\item[(ii)] each row has strictly shorter length than its predecessor.
\end{itemize}
We denote the depth $l$ by $l(\lambda)$ and we denote $\lambda \vdash n$ if $\lambda$ consists of $n$ boxes.
We write $(i,j) \in \lambda$ if there exists a box in the $i$th row and the $j$th column.
We identify a shifted Young diagram $\lambda$ with the strict partition $(\lambda_1 > \lambda_2 > \lambda_3 > \cdots )$, where
$\lambda_i$ is the number of boxes in the $i$th row of $\lambda$. For $\lambda=(\lambda_1> \lambda_2> \cdots )\vdash n$ and
$\mu=(\mu_1> \mu_2> \cdots )\vdash n$, we say $\lambda$ \emph{dominates} $\mu$ and write $\lambda \trianglerighteq \mu$ if
\begin{align} \label{Eq: order on SYD}
 \sum_{i=1}^k \lambda_i \ge \sum_{i=1}^k \mu_i, \text{ for }\;1\le k\le l(\lambda).
\end{align}

A {\it standard tableau} $T$ of shape $\lambda\vdash n$ is
a filling of $n$ boxes of $\lambda$ with numbers $1$ to $n$ such that (i) each number is used exactly once, (ii) the numbers in rows
and columns increase from left to right and top to bottom, respectively. Let $\ST(\lambda)$ be the set of all standard tableaux of shape $\lambda$. For example,
there are 3 standard tableaux of shape $(4,1)$:
\vskip 0.3em
$$
\xy
(0,12)*{};(24,12)*{} **\dir{-};  
(0,6)*{};(24,6)*{} **\dir{-};
(6,0)*{};(12,0)*{} **\dir{-};
(0,6)*{};(0,12)*{} **\dir{-};
(6,0)*{};(6,12)*{} **\dir{-};
(12,0)*{};(12,12)*{} **\dir{-};
(18,6)*{};(18,12)*{} **\dir{-};
(24,6)*{};(24,12)*{} **\dir{-};
(3,9)*{1}; (9,9)*{2}; (15,9)*{3}; (21,9)*{4};
(9,3)*{5};
(35,12)*{};(59,12)*{} **\dir{-}; 
(35,6)*{};(59,6)*{} **\dir{-};
(41,0)*{};(47,0)*{} **\dir{-};
(35,6)*{};(35,12)*{} **\dir{-};
(41,0)*{};(41,12)*{} **\dir{-};
(47,0)*{};(47,12)*{} **\dir{-};
(53,6)*{};(53,12)*{} **\dir{-};
(59,6)*{};(59,12)*{} **\dir{-};
(38,9)*{1}; (44,9)*{2}; (50,9)*{3}; (56,9)*{5};
(44,3)*{4};
(70,12)*{};(94,12)*{} **\dir{-}; 
(70,6)*{};(94,6)*{} **\dir{-};
(76,0)*{};(82,0)*{} **\dir{-};
(70,6)*{};(70,12)*{} **\dir{-};
(76,0)*{};(76,12)*{} **\dir{-};
(82,0)*{};(82,12)*{} **\dir{-};
(88,6)*{};(88,12)*{} **\dir{-};
(94,6)*{};(94,12)*{} **\dir{-};
(73,9)*{1}; (79,9)*{2}; (85,9)*{4}; (91,9)*{5};
(79,3)*{3};
(27,6)*{,}; (62,6)*{,}; (97,6)*{.};
\endxy
$$

For $\lambda = (\lambda_1 > \lambda_2 > \cdots ) \vdash n$, let $T^{\lambda}$ be the tableau of shape $\lambda$ whose $(i,j)$-entry is
$$(j-i)+1+\sum_{k=1}^{i-1} \lambda_k. $$ We call  $T^{\lambda}$ the {\it canonical tableau} of shape $\lambda$.
In the above example, the first tableau is the canonical tableau of shape $(4,1)$.
The following theorem is well-known.

\begin{thm}[\protect{cf. \cite[Thm.2.2.1,\ Cor.3.2.2]{Sagan90}}] \label{Thm: hook formula}
\begin{itemize}
\item[(1)] For $\lambda \vdash n$, the number of standard tableaux of shape $\lambda$ is given by the following hook formula.
$$ | \ST(\lambda)| = \frac{ n! }{ \prod_{(i,j)\in \lambda} h_{i,j} }, $$
where $h_{i,j}$ is the hook-length of $(i,j)\in\lambda$, i.e.,
$$ h_{i,j} =  | \{ (i,j') \in \lambda \mid j' \ge j \}| + | \{ (i',j) \in \lambda \mid i' > i \}| + | \{ (j+1,j') \in \lambda \mid j' > j \}|. $$

\item[(2)] We have the following equality.
$$ n! = \sum_{\lambda \vdash n} 2^{ n - l(\lambda)} |\ST(\lambda)|^2. $$
\end{itemize}
\end{thm}

\vskip 1em

Let $\lambda \vdash n$ be a shifted Young diagram. We declare that each row has the residue pattern $$ 0\; 1\; 2\; \cdots\; \ell\; \cdots\; 2\; 1\; 0 ,$$ which repeats from left to right in each row, and define $\res(i,j)$ to be the residue of $(i,j)\in\lambda$. For example, if $\g$ is of type $A_{6}^{(2)}$ and $\lambda = (10,6,3,1)$, the residues are as follows:
$$
\xy
(0,12)*{};(60,12)*{} **\dir{-};
(0,6)*{};(60,6)*{} **\dir{-};
(6,0)*{};(42,0)*{} **\dir{-};
(12,-6)*{};(30,-6)*{} **\dir{-};
(18,-12)*{};(24,-12)*{} **\dir{-};
(0,12)*{};(0,6)*{} **\dir{-};
(6,12)*{};(6,0)*{} **\dir{-};
(12,12)*{};(12,-6)*{} **\dir{-};
(18,12)*{};(18,-12)*{} **\dir{-};
(24,12)*{};(24,-12)*{} **\dir{-};
(30,12)*{};(30,-6)*{} **\dir{-};
(36,12)*{};(36,0)*{} **\dir{-};
(42,12)*{};(42,0)*{} **\dir{-};
(48,12)*{};(48,6)*{} **\dir{-};
(54,12)*{};(54,6)*{} **\dir{-};
(60,12)*{};(60,6)*{} **\dir{-};
(3,9)*{0}; (9,9)*{1}; (15,9)*{2}; (21,9)*{3}; (27,9)*{2}; (33,9)*{1}; (39,9)*{0};(45,9)*{0};(51,9)*{1}; (57,9)*{2};
           (9,3)*{0}; (15,3)*{1}; (21,3)*{2}; (27,3)*{3}; (33,3)*{2}; (39,3)*{1};
                      (15,-3)*{0}; (21,-3)*{1}; (27,-3)*{2};
                                   (21,-9)*{0};
\endxy
$$
Thus, if we pick $(2,5)\in\lambda$ then it has residue $3$.  Now, we are ready to introduce the residue sequence of a tableau.

\begin{defn}
Let $T \in \ST(\lambda)$. We define the \emph{residue sequence} of $T$ by
$$ \res(T) = (\res(i_1,j_1), \res(i_2,j_2), \ldots, \res(i_n,j_n) )\in I^n,$$
where $(i_k,j_k)\in\lambda$ is the box filled with $k$ in $T$, for $1\leq k\leq n$.
\end{defn}

We remark that the residue pattern for shifted Young diagrams is the same as the pattern of Young walls defined by the level one perfect crystal of type $A_{2\ell}^{(2)}$
\cite{HK02, Kang03, KKMMNN91, KKMMNN92}, which we will use in Section \ref{Subsection: crystals}.  It reminds the reader of
the $q$-deformed Fock space introduced in \cite{KangKwon08}, and of the $q$-deformed Fock space from \cite{KMPY96}.
However, we must distinguish between categorification of integrable highest weight modules and its crystals.
It seems more natural to consider the embedding of
the basic module $V(\Lambda_0)$ into the classical Fock space which we will recall in Section \ref{Subsection: Fock space}, rather than the $q$-deformed Fock space.
We suspect that it would be related to quasihereditary covers of finite quiver Hecke algebras and some theory of packets for those algebras would explain the distinction of the classical Fock space from the deformed Fock space.

\vskip 1em

\subsection{Cartan datum}\label{Cartan datum}

Let $I = \{0,1, \ldots, \ell \}$, for $\ell \ge 2$, and let $\cmA$ be the {\it affine Cartan matrix} of type $A_{2\ell}^{(2)}$, i.e.,
$$ \cmA = (a_{ij})_{i,j\in I} = \left(
                                  \begin{array}{ccccccc}
                                    2  & -2 & 0  & \ldots & 0  & 0  & 0 \\
                                    -1 &  2 & -1 & \ldots & 0  & 0  & 0 \\
                                    0  & -1 & 2  & \ldots & 0  & 0  & 0 \\
                                    \vdots   &  \vdots  &  \vdots  & \ddots &  \vdots  &  \vdots  & \vdots \\
                                    0  & 0  & 0  & \ldots & 2  & -1 & 0 \\
                                    0  & 0  & 0  & \ldots & -1 & 2  & -2 \\
                                    0  & 0  & 0  & \ldots & 0  & -1 & 2 \\
                                  \end{array}
                                \right).
  $$
When $\ell=1$, the affine Cartan matrix of type $A^{(2)}_2$ is $$\cmA = (a_{ij})_{i,j\in I} = \begin{pmatrix} 2  & -4 \\ -1 & 2 \end{pmatrix}.$$
We choose a realization of the Cartan matrix and obtain
an {\it affine Cartan datum} $(\cmA, \wlP, \Pi, \Pi^{\vee})$, where
\begin{itemize}
\item[(1)] $\cmA$ is the affine Cartan matrix of type $A_{2\ell}^{(2)}$ as above,
\item[(2)] $\wlP$ is a free abelian group of rank $\ell+2$, called the {\it weight lattice},
\item[(3)] $\Pi = \{ \alpha_i \mid i\in I \} \subset \wlP$, called the set of {\it simple roots},
\item[(4)] $\Pi^{\vee} = \{ h_i \mid i\in I\} \subset \wlP^{\vee} := \Hom( \wlP, \Z )$, called the set of {\it simple coroots},
\end{itemize}
which satisfy the following properties:
\begin{itemize}
\item[(a)] $\langle h_i, \alpha_j \rangle  = a_{ij}$ for all $i,j\in I$,
\item[(b)] $\Pi$ and $\Pi^{\vee}$ are linearly independent, respectively.
\end{itemize}
The weight lattice $\wlP$ has a symmetric bilinear pairing $(\ | \ )$ satisfying
$$ ( \alpha_i | \Lambda ) = d_i \langle h_i , \Lambda \rangle \ \text{ for all } \Lambda \in \wlP,$$
where $(d_0,d_1, \ldots, d_\ell) = (1,2, \ldots, 2,4)$.
We write $ \Lambda(h) = \langle h, \Lambda \rangle $ for $h \in \wlP^\vee,\ \Lambda \in \wlP$, and denote by
$$\wlP^+ = \{ \Lambda \in \wlP \mid \Lambda(h_i) \in \Z_{\ge 0},\ i\in I \}$$
the
set of {\it dominant integral weights}. We fix an element $d$ which satisfies $\langle d, \alpha_i \rangle=\delta_{i0}$.
It is called the \emph{scaling element},
and we assume that $\Pi^{\vee}\sqcup\{d\}$ is a $\Z$-basis of $P^{\vee}$ as in \cite[p.21]{HK02}.
Then,
for $i\in I$, we define the $i$th {\it fundamental weight} $\Lambda_i \in \wlP^+$, i.e. the  element defined by $\Lambda_i(d)=0$ and $\Lambda_i(h_j) = \delta_{ij}$, for $j\in I$.  The free abelian group $\rlQ = \bigoplus_{i \in I} \Z \alpha_i$ is called the {\it root lattice},  and $\rlQ^+ = \sum_{i\in I} \Z_{\ge 0} \alpha_i$ is
the {\it positive cone} of the root lattice. For $\beta=\sum_{i \in I} k_i \alpha_i \in \rlQ^+$, the {\it height} of $\beta$ is defined by $|\beta|=\sum_{i \in I} k_i$.

Let $\weyl=\langle r_i \mid i\in I\rangle$ be the {\it Weyl group} associated with $\cmA$. The Coxeter generators $\{r_i\}_{i\in I}$ act on $P$ by
$r_i\Lambda=\Lambda-\langle h_i, \Lambda\rangle\alpha_i$, for $\Lambda\in P$, as usual. Finally, the \emph{null root} in type $A^{(2)}_{2\ell}$ is
$$ \delta = 2\alpha_0 + 2\alpha_1 + \cdots + 2\alpha_{\ell-1} + \alpha_\ell, $$
for $\ell\ge1$. Note that $ \langle h_i, \delta\rangle = 0$ and $w \delta = \delta$, for $i\in I$ and $w\in \weyl$. The residue pattern  for shifted Young diagrams comes from the null root, and we denote
the residue pattern $01\ldots\ell\ldots10$ by $\nu_\delta$.

\vskip 1em

\subsection{The Fock space of neutral fermions}\label{Subsection: Fock space}

Recall that $\g(\cmA)$ is realized as $BKP_{2\ell+1}$ in the terminology of \cite[\S2]{DJKM82}. In other words, $\g(\cmA)$ is obtained by
$(2\ell+1)$-reduction from $\g(B_\infty)$. To obtain the basic $\g(\cmA)$-module $V(\Lambda_0)$, we start with the basic $\g(B_\infty)$-module.
Thus, we recall the Fock space of neutral fermions from \cite[\S6]{JM83}, which affords the basic spin representation of $\g(B_\infty)$.
Let $\mathsf{C}$ be the Clifford algebra over complex numbers $\C$ defined by generators $\phi_k$ ($k\in \Z$)
and the following relations:
$$ \phi_p \phi_q + \phi_q \phi_p = \left\{
                                     \begin{array}{ll}
                                       2 & \hbox{ if }  p=q=0, \\
                                       (-1)^{p} \delta_{p, -q} & \hbox{ otherwise.}
                                     \end{array}
                                   \right.
$$
The Chevalley generators of $\g(B_\infty)$ are given as follows:
\begin{align*}
\mathsf{e}_0 &=  \phi_{-1} \phi_{0}, \qquad \quad \ \  \mathsf{e}_j = (-1)^j \phi_{-j-1} \phi_{j} \qquad  \qquad  \qquad \qquad  (j\ge 1), \\
\mathsf{f}_0 &= \phi_{1} \phi_{0}, \qquad  \quad\ \ \ \ \mathsf{f}_j = (-1)^j \phi_{j+1} \phi_{-j} \qquad \qquad  \qquad \qquad  (j\ge 1), \\
\mathsf{h}_0 &= 2 \phi_{1} \phi_{-1}+1, \quad  \ \mathsf{h}_j = (-1)^j (\phi_{j} \phi_{-j} + \phi_{j+1} \phi_{-j-1}) \qquad \ (j\ge 1).
\end{align*}
Let $\mathsf{I}$ be the left ideal of  $\mathsf{C}$ generated by \{$\phi_k \mid k < 0\}$, and define $ \mathsf{F} = \mathsf{C} / \mathsf{I}$.
Then, $ \mathsf{F} $ is the direct sum of two irreducible $\g(B_\infty)$-modules
$\mathsf{F}_0$ and $\mathsf{F}_1$ with highest weights $\Lambda_0$,
and their highest weight vectors are
$|0\rangle := 1+\mathsf{I}$ and $|1\rangle :=\phi_0|0\rangle$, respectively.
We take a shifted Young diagram $\lambda$. It can be written as $\lambda = (\lambda_1 > \lambda_2 > \ldots > \lambda_{2r-1}> \lambda_{2r} \ge 0)$
for a unique $r$, where $\lambda_{2r}>0$ if $l(\lambda)$ is even and $\lambda_{2r}=0$ if $\ell(\lambda)$ is odd. We define
$|\lambda\rangle = \phi_{\lambda_1} \phi_{\lambda_2} \cdots  \phi_{\lambda_{2r}} |0\rangle. $ Then, they are linearly independent and
$$ \F:= \mathsf{F}_0 = \Span_\C \{ |\lambda\rangle \mid \lambda : \text{shifted Young diagrams}  \}. $$
It is easy to see that if $\lambda$ has a row of length $j+1$ and does not have a row of length $j$ (resp. $\lambda$ has a row of length $j$ and does not have
a row of length $j+1$), then we have
\begin{align} \label{Eq: action of efh for Binf}
\mathsf{e}_j |\lambda\rangle = |\mathsf{e}_j\lambda\rangle, \quad (\text{resp.}\;\;\mathsf{f}_j |\lambda\rangle = |\mathsf{f}_j\lambda\rangle\;),
\end{align}
where $\mathsf{e}_j\lambda$ (resp.\ $\mathsf{f}_j\lambda$) is the shifted Young diagram obtained from $\lambda$ by deleting the rightmost box of
the row of length $j+1$ (resp. adding a new box on the right of the row of length $j$).
Otherwise, we have $\mathsf{e}_j |\lambda\rangle=0$ (resp. $ \mathsf{f}_j |\lambda\rangle=0)$. Therefore, we may think of $\F$ as a based vector space
whose basis is given by $|\lambda\rangle$'s, and it is an integrable $\g(B_\infty)$-module by the action
$|\lambda\rangle\mapsto|\mathsf{e}_j\lambda\rangle, |\mathsf{f}_j\lambda\rangle$. This is our \emph{Fock space}.

We now define a $\g$-module structure on the space $\mathcal{F}$ by reduction. Let $h=2\ell+1$.
>From \cite[Table $2$]{DJKM82}, the action of the Chevalley generators $f_i, e_i$ of $\g$
on $\F$ are given as follows:
\begin{align} \label{Eq: Chevalley generators}
f_i = \sum_{j\ge0, \ j\equiv i, -i-1} \mathsf{f}_j,  \qquad
e_i = \left\{
        \begin{array}{ll}
          \mathsf{e}_0 + 2 \sum_{j>0, \  j\equiv 0, -1} \mathsf{e}_j & \hbox{ if } i=0, \\
          \sum_{j>0, \ j\equiv i, -i-1} \mathsf{e}_j & \hbox{ if } i = 1, \ldots, \ell,
        \end{array}
      \right.
\end{align}
where all congruences are taken modulo $h$. Each $|\lambda\rangle$ is a weight vector, and if we define a multiset $\res(\lambda)$ of $I$ by $\res(\lambda)=\{ \res(i,j) \mid (i,j)\in\lambda\}$, then its weight is given by
$$ \wt(\lambda) = \Lambda_0-\sum_{k\in\res(\lambda)} \alpha_k. $$

We remark that $\mathcal{F}$ is isomorphic to the direct sum of countably many copies of $V(\Lambda_0)$.

\section{Cyclotomic quiver Hecke algebras of type $A^{(2)}_{2\ell}$}  

Let $(A,P,P^{\vee},\Pi,\Pi^{\vee})$ be a Cartan datum associated with a symmetrizable Cartan matrix $A$.
In this section, we review results on categorification of integrable $U_q(\mathfrak{g}(A))$-modules and their crystals.
As was explained in the introduction, this categorification gives us a new family of self-injective algebras, namely
cyclotomic quiver Hecke algebras, and the algebras are main object of study in this paper.
Categorification of the basic integrable module specialized at $q=1$ combined with Lie theoretic treatment of the Fock space explained
in the previous section are key ingredients in discovering dimension formulas for finite quiver Hecke algebras on the one hand,
categorification of its crystal gives us the number of irreducible modules of cyclotomic quiver Hecke algebras and
their behavior under induction and restriction functors. For the latter, realization of the abstract crystal in an explicit combinatorial model
is required, and we use Young walls for the purpose. The results in this section will be used in a crucial manner in the proof of
our main theorems in later sections.

\subsection{Quiver Hecke algebras}

Throughout the paper, $\bR$ is an algebraically closed field, and algebras are unital associative  $\bR$-algebras.

Let $(\cmA, \wlP, \Pi, \Pi^{\vee})$ be the affine Cartan datum from Section \ref{Cartan datum}. We are going to define $\bR$-algebras which we call
\emph{finite quiver Hecke algebras of type $A^{(2)}_{2\ell}$}. To do this, we choose polynomials $\mathcal{Q}_{i,j}(u,v)\in\bR[u,v]$, for $i,j\in I$,
of the form
\begin{align*}
\mathcal{Q}_{i,j}(u,v) = \left\{
                 \begin{array}{ll}
                   \sum_{p(\alpha_i|\alpha_i)+q (\alpha_j|\alpha_j) + 2(\alpha_i|\alpha_j)=0} t_{i,j;p,q} u^pv^q & \hbox{if } i \ne j,\\
                   0 & \hbox{if } i=j,
                 \end{array}
               \right.
\end{align*}
where $t_{i,j;p,q} \in \bR$ are such that $t_{i,j;-a_{ij},0} \ne 0$ and $\mathcal{Q}_{i,j}(u,v) = \mathcal{Q}_{j,i}(v,u)$.
For example, if $\ell \ge 2$, $\mathcal{Q}_{i,j}(u,v)$ should have the following form:
$$
\mathcal{Q}_{i,j}(u,v) = \left\{
                 \begin{array}{ll}
                   t_{i,j; 0, 0} & \hbox{if } a_{ij}= a_{ji} = 0,\\
                   t_{i,j; 1,0}\ u + t_{i,j;0,1}\ v  & \hbox{if } a_{ij}= a_{ji} = -1,\\
                   t_{i,j; 2,0}\ u^2 + t_{i,j;0,1}\ v  & \hbox{if } a_{ij}=-2, a_{ji} = -1,\\
                   t_{i,j; 1,0}\ u + t_{i,j;0,2}\ v^2  & \hbox{if } a_{ij}=-1, a_{ji} = -2,\\
                   0 & \hbox{if } i=j,
                 \end{array}
               \right.
$$
where $t_{i,j; p,q} \ne 0$ and $ t_{i,j; p,q} = t_{j,i; q,p}$ for all $i,j,p,q$.

The symmetric group $\sg_n = \langle s_k \mid k =1, \ldots, n-1\rangle$ acts on $I^n$ by place permutations. Namely, the Coxeter generator $s_k$ acts on $I^n$ by
$$s_k(\nu_1,\dots,,\nu_k,\nu_{k+1},\dots,\nu_n)=(\nu_1,\dots,\nu_{k+1},\nu_k,\dots,\nu_n).$$

\begin{defn} \
Let $\Lambda \in \wlP^+$. The {\it cyclotomic quiver Hecke algebra} $R^{\Lambda}(n)$ associated with polynomials $(\mathcal{Q}_{i,j}(u,v))_{i,j\in I}$ and the dominant integral weight $\Lambda$
is the $\Z$-graded  $\bR$-algebra defined by three sets of generators
$$\{e(\nu) \mid \nu = (\nu_1,\ldots, \nu_n) \in I^n\}, \;\{x_k \mid 1 \le k \le n\}, \;\{\psi_l \mid 1 \le l \le n-1\} $$
subject to the following relations:

\begin{align*}
& e(\nu) e(\nu') = \delta_{\nu,\nu'} e(\nu),\ \sum_{\nu \in I^{n}} e(\nu)=1,\
x_k e(\nu) =  e(\nu) x_k, \  x_k x_l = x_l x_k,\\
& \psi_l e(\nu) = e(s_l(\nu)) \psi_l,\  \psi_k \psi_l = \psi_l \psi_k \text{ if } |k - l| > 1, \\[5pt]
&  \psi_k^2 e(\nu) = \mathcal{Q}_{\nu_k, \nu_{k+1}}(x_k, x_{k+1}) e(\nu), \\[5pt]
&  (\psi_k x_l - x_{s_k(l)} \psi_k ) e(\nu) = \left\{
                                                           \begin{array}{ll}
                                                             -  e(\nu) & \hbox{if } l=k \text{ and } \nu_k = \nu_{k+1}, \\
                                                               e(\nu) & \hbox{if } l = k+1 \text{ and } \nu_k = \nu_{k+1},  \\
                                                             0 & \hbox{otherwise,}
                                                           \end{array}
                                                         \right. \\[5pt]
&( \psi_{k+1} \psi_{k} \psi_{k+1} - \psi_{k} \psi_{k+1} \psi_{k} )  e(\nu) \\[4pt]
&\qquad \qquad \qquad = \left\{
                                                                                   \begin{array}{ll}
\displaystyle \frac{\mathcal{Q}_{\nu_k,\nu_{k+1}}(x_k,x_{k+1}) -
\mathcal{Q}_{\nu_k,\nu_{k+1}}(x_{k+2},x_{k+1})}{x_{k}-x_{k+2}} e(\nu) & \hbox{if } \nu_k = \nu_{k+2}, \\
0 & \hbox{otherwise}, \end{array}
\right.\\[5pt]
& x_1^{\langle h_{\nu_1}, \Lambda \rangle} e(\nu)=0.
\end{align*}
\end{defn}

\bigskip
The $\Z$-grading on $R^\Lambda(n)$ is given as follows:
\begin{align*}
\deg(e(\nu))=0, \quad \deg(x_k e(\nu))= ( \alpha_{\nu_k} |\alpha_{\nu_k}), \quad  \deg(\psi_l e(\nu))= -(\alpha_{\nu_{l}} | \alpha_{\nu_{l+1}}).
\end{align*}

\vskip 1em

The following statement was proved in a special case \cite[Lem.2.1]{BK09}. As we are not able to find a reference for general case, we add a proof, but it is straightforward.

\begin{lemma}
The algebra $R^\Lambda(n)$ is a finite dimensional algebra and $x_1,\dots,x_n$ are nilpotent.
\end{lemma}
\begin{proof}
It is proved in \cite[Cor.4.4]{KK11} that $R^\Lambda(n)$ is a finite dimensional algebra. Thus, we prove that $x_1,\dots,x_n$ are nilpotent elements. Note that, by general theory of graded Artin algebras,
irreducible $R^\Lambda(n)$-modules are gradable, and it implies that any element of positive degree acts nilpotently on finite length modules. 
\end{proof}

\vskip 1em

For $\beta \in \mathsf{Q}^+$ with $|\beta|=n$, let
$$ I^\beta = \{ \nu=(\nu_1, \ldots, \nu_n) \in I^n \mid \alpha_{\nu_1} + \cdots + \alpha_{\nu_n} = \beta \}.$$
The symmetric group action preserves $I^\beta$, so that  $$ e(\beta) = \sum_{\nu \in I^\beta} e(\nu)$$ is a central idempotent of $R^\Lambda(n)$.
We define $$R^\Lambda(\beta) = R^\Lambda(n) e(\beta).$$
The following algebras are the object of study in this paper. We do not know whether $R^\Lambda(\beta)$ is an indecomposable $\bR$-algebra, for every $\beta$.

\begin{defn}
We call the algebras $\KLR(\beta)$ \emph{finite quiver Hecke algebras of type $A^{(2)}_{2\ell}$}.
\end{defn}

\vskip 1em

\subsection{Categorification of integrable modules} \label{Subsection: integrable modules}

Let $\g=\g(\cmA)$ be the affine Kac-Moody Lie algebra associated with the Cartan matrix $\cmA$. Let $q$ be an indeterminate and we denote the corresponding quantum affine algebra by $U_q(\g)$. It is an  $\Q(q)$-algebra defined by generators $e_i,f_i$ $(i \in I)$, $q^{h}$ $(h \in \mathsf{P}^{\vee})$ and their relations.


For $i\in I$, let $q_i = q^{d_i}$ and
\begin{equation*}
 \begin{aligned}
 \ &[n]_i =\frac{ q^n_{i} - q^{-n}_{i} }{ q_{i} - q^{-1}_{i} },
 \ &[n]_i! = \prod^{n}_{k=1} [k]_i.
 \end{aligned}
\end{equation*}

For $\A =\Z[q,q^{-1}]$, we consider the $\A$-subalgebra of $U_q(\g)$ generated by $f_i^{(k)} := f_i^{k}/[k]_i! $, for $i\in I$ and $k\in\Z_{\ge0}$. We
denote it by $U_{\A}^-(\g)$.

Let  $\Lambda \in \wlP^+$ be a dominant integral weight. Then we may consider the irreducible highest weight $U_q(\g)$-module $V_q(\Lambda)$ with highest weight $\Lambda$
and its $\A$-form $V_\A(\Lambda)$, which is the $U_{\A}^-(\g)$-submodule of $V_q(\Lambda)$ generated by a fixed highest weight vector. If we specialize $V_\A(\Lambda)$ at $q=1$, we have
the irreducible highest weight $\g$-module $V(\Lambda)$.

We denote the direct sum of split Grothendieck groups of additive categories $R^\Lambda(\beta)\text{-}\proj^\Z$ of finitely generated projective
graded left $R^\Lambda(\beta)$-modules, for $\beta\in \rlQ^+$, by
\begin{align*}
K^\Z_0(R^{\Lambda})=\bigoplus_{\beta \in \rlQ^+} K_0(R^{\Lambda}(\beta)\text{-}\proj^\Z).
\end{align*}
It has an $\A$-module structure induced by the $\Z$-grading on $R^\Lambda(n)$.

Let $e(\nu,\nu')$ be the idempotent corresponding to the concatenation $\nu*\nu'$ of
$\nu$ and $\nu'$. If $\nu'=i$, we write $e(\nu,i)$.
For $\beta\in \rlQ^+$ and $i\in I$, we set
$$e(\beta, i) = \sum_{\nu \in I^\beta} e(\nu, i)$$
and define functors
\begin{align*}
E_i &: R^\Lambda(\beta + \alpha_i)\text{-mod}^\Z \longrightarrow R^\Lambda(\beta)\text{-mod}^\Z, \\
F_i &: R^\Lambda(\beta )\text{-mod}^\Z \longrightarrow R^\Lambda(\beta+ \alpha_i)\text{-mod}^\Z,
\end{align*}
between categories $R^\Lambda(\beta)\text{-mod}^\Z$ and
$R^\Lambda(\beta + \alpha_i)\text{-mod}^\Z$ of finitely generated graded modules
by $E_i(N)  = e(\beta,i)N$ and $F_i(M) = R^\Lambda(\beta+ \alpha_i) e(\beta,i) \otimes_{R^\Lambda(\beta)}M$, for $M \in R^\Lambda(\beta)\text{-mod}^\Z$
and $N \in R^\Lambda(\beta + \alpha_i)\text{-mod}^\Z$, respectively. Then, $E_i$ and $F_i$ are exact functors \cite[Thm.4.5]{KK11} which send projective modules to projective modules, and the following theorem holds.

\begin{thm} [\protect{\cite[Thm.5.2]{KK11}}]  \label{Thm: categorification}
Set $l_i = \langle h_i, \Lambda - \beta  \rangle$, for $i\in I$. Then one of the following  isomorphisms of endofuctors on
the category of finitely generated graded $R^{\Lambda}(\beta)$-modules holds.
\begin{enumerate}
\item If $l_i \ge 0$, then
$$ q_i^{-2}F_i E_i \oplus \bigoplus_{k=0}^{l_i- 1} q_i^{2k} \mathrm{id} \buildrel \sim \over \longrightarrow E_iF_i .$$
\item If $l_i \le 0$, then
$$ q_i^{-2}F_i E_i  \buildrel \sim \over \longrightarrow  E_iF_i \oplus \bigoplus_{k=0}^{-l_i- 1} q_i^{-2k-2}  \mathrm{id} .$$
\end{enumerate}
\end{thm}

Thus, functors $ q_i^{1-\langle h_i, \Lambda - \beta \rangle} E_i$ and $F_i$, for $\beta\in \rlQ^+$ and $i\in I$, define a $U_{\A}(\g)$-module structure on $K^\Z_{0}(R^\Lambda)$.

The following theorem is the cyclotomic categorification theorem conjectured by Khovanov and Lauda, and proved by Kang and Kashiwara.

\begin{thm}[\protect{\cite[Thm.6.2]{KK11}}] \label{Thm: categorification V}
There exists a $U_\A(\g)$-module isomorphism between $K^\Z_0(R^{\Lambda})$ and $V_\A(\Lambda)$.
\end{thm}

In the study of representation type, we do not need grading. So we specialize $q\to1$. Let $R^{\Lambda}(\beta)\text{-}\proj$ be
the category of finitely generated projective $R^{\Lambda}(\beta)$-modules, and define
\begin{align*}
K_0(R^{\Lambda})=\bigoplus_{\beta \in \rlQ^+} K_0(R^{\Lambda}(\beta)\text{-}\proj).
\end{align*}
Then we have an isomorphism of $\g_\Z$-modules $K_0(R^{\Lambda})\cong V(\Lambda)_\Z$, where $\g_\Z$ and $V(\Lambda)_\Z$ are the Kostant $\Z$-form of $\g$ and $V(\Lambda)$, respectively.

\subsection{Categorification of crystals}\label{Subsection: crystals}

We now recall a combinatorial realization of the crystal of $V_q(\Lambda_0)$ by {\it Young walls} \cite{HK02, Kang03}.
A Young wall is a generalization of a colored Young diagram based on certain level 1 perfect crystal, which gives a combinatorial realization of crystals for basic representations of various quantum affine algebras. This realization will be used in Proposition \ref{Prop: the number of simples} for counting irreducible modules over finite quiver Hecke algebras. We refer the reader to \cite{HK02, Kang03, Kash91, Kash93} for the details. For type $A^{(2)}_{2\ell}$, we consider the following blocks:
\begin{align*}
\xy
(0,1.5)*{};(6,1.5)*{} **\dir{-};
(0,-1.5)*{};(6,-1.5)*{} **\dir{-};
(0,-1.5)*{};(0,1.5)*{} **\dir{-};
(6,-1.5)*{};(6,1.5)*{} **\dir{-};
(3.2,0)*{ _0};
(50,0)*{ : \text{ unit width and half-unit height, unit thickness}};
(0,-4)*{};(6,-4)*{} **\dir{-};
(0,-10)*{};(6,-10)*{} **\dir{-};
(0,-4)*{};(0,-10)*{} **\dir{-};
(6,-4)*{};(6,-10)*{} **\dir{-};
(3.2,-7)*{ _i};
(55,-7)*{ (i=1,\ldots, \ell): \text{ unit width and unit height, unit thickness}};
\endxy
\end{align*}
and define a {\it Young wall} to be a wall consisting of the above colored blocks stacked by the following rules. We write (2) for completeness, but it is a vacant condition, as we do not have blocks of half-unit thickness.
\begin{enumerate}
\item Colored blocks should be stacked in the pattern given below.
\item No block can be placed on top of a column of half-unit thickness.
\item Except for the rightmost column, there should be no free space to the right of any block.
\end{enumerate}
The pattern is given as follows:
$$
\xy
(0,-0.5)*{};(33,-0.5)*{} **\dir{.};
(0,-1)*{};(33,-1)*{} **\dir{.};
(0,-1.5)*{};(33,-1.5)*{} **\dir{.};
(0,-2)*{};(33,-2)*{} **\dir{.};
(0,-2.5)*{};(33,-2.5)*{} **\dir{.};
(0,-3)*{};(33,-3)*{} **\dir{-};
(0,0)*{};(33,0)*{} **\dir{-};
(0,3)*{};(33,3)*{} **\dir{-};
(0,9)*{};(33,9)*{} **\dir{-};
(0,19)*{};(33,19)*{} **\dir{-};
(0,25)*{};(33,25)*{} **\dir{-};
(0,35)*{};(33,35)*{} **\dir{-};
(0,41)*{};(33,41)*{} **\dir{-};
(0,44)*{};(33,44)*{} **\dir{-};
(0,47)*{};(33,47)*{} **\dir{-};
(0,53)*{};(33,53)*{} **\dir{-};
(33,-3)*{};(33,56)*{} **\dir{-};
(27,-3)*{};(27,56)*{} **\dir{-};
(21,-3)*{};(21,56)*{} **\dir{-};
(15,-3)*{};(15,56)*{} **\dir{-};
(9,-3)*{};(9,56)*{} **\dir{-};
(3,-3)*{};(3,56)*{} **\dir{-};
(30.2,1.5)*{_0}; (30.2,6)*{_1}; (30.2,15)*{\vdots}; (30.2,22)*{_\ell}; (30.2,31)*{\vdots};
(30.2,38)*{_1}; (30.2,42.5)*{_0}; (30.2,45.5)*{_0}; (30.2,50)*{_1};
(24.2,1.5)*{_0}; (24.2,6)*{_1}; (24.2,15)*{\vdots}; (24.2,22)*{_\ell}; (24.2,31)*{\vdots};
(24.2,38)*{_1}; (24.2,42.5)*{_0}; (24.2,45.5)*{_0}; (24.2,50)*{_1};
(18.2,1.5)*{_0}; (18.2,6)*{_1}; (18.2,15)*{\vdots}; (18.2,22)*{_\ell}; (18.2,31)*{\vdots};
(18.2,38)*{_1}; (18.2,42.5)*{_0}; (18.2,45.5)*{_0}; (18.2,50)*{_1};
(12.2,1.5)*{_0}; (12.2,6)*{_1}; (12.2,15)*{\vdots}; (12.2,22)*{_\ell}; (12.2,31)*{\vdots};
(12.2,38)*{_1}; (12.2,42.5)*{_0}; (12.2,45.5)*{_0}; (12.2,50)*{_1};
(6.2,1.5)*{_0}; (6.2,6)*{_1}; (6.2,15)*{\vdots}; (6.2,22)*{_\ell}; (6.2,31)*{\vdots};
(6.2,38)*{_1}; (6.2,42.5)*{_0}; (6.2,45.5)*{_0}; (6.2,50)*{_1};
(35,-3)*{.};
\endxy
$$
The sequence $(0,1,2, \ldots, \ell-1, \ell, \ell-1, \ldots, 2,1,0 )$ of colors is repeated in each column.
For a Young wall $Y$, define
$$ \wt(Y) = \Lambda_0 - \sum_{i\in I} k_i \alpha_i,  $$
where $k_i$ is the number of $i$-blocks in $Y$, for $i\in I$.

A column is called a {\it full column} if its height is a multiple of the unit length and its top is of unit thickness.
As the ground state blocks have half-unit height, a column is not full if and only if the number of stacked blocks is divisible by $|\delta|=2\ell+1$.
A Young wall is said to be {\it proper} if none of the full columns have the same height.
A part of a column consisting of two $0$-blocks, two $1$-blocks, \ldots, two $(\ell-1)$-blocks and one $\ell$-block
is called a $\delta$-column. A column in a proper Young wall is said to contain a {\it removable} $\delta$ if we may remove
a $\delta$-column from $Y$ and still obtain a proper Young wall. For example, when $A $ is of type $ A_{4}^{(2)}$, i.e. $\ell=2$, the following are Young walls.
$$
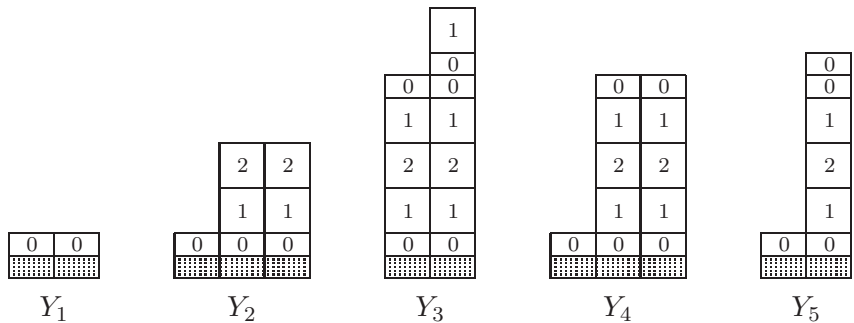
\xy
(0,0)*{};(12,0)*{} **\dir{-};   
(0,3)*{};(12,3)*{} **\dir{-};
(0,6)*{};(12,6)*{} **\dir{-};
(0,0)*{};(0,6)*{} **\dir{-};
(6,0)*{};(6,6)*{} **\dir{-};
(12,0)*{};(12,6)*{} **\dir{-};
(0,0.5)*{};(12,0.5)*{} **\dir{.};
(0,1)*{};(12,1)*{} **\dir{.};
(0,1.5)*{};(12,1.5)*{} **\dir{.};
(0,2)*{};(12,2)*{} **\dir{.};
(0,2.5)*{};(12,2.5)*{} **\dir{.};
(3.2,4.5)*{_0}; (9.2,4.5)*{_0}; (6,-4)*{Y_1};
(22,0)*{};(40,0)*{} **\dir{-};   
(22,3)*{};(40,3)*{} **\dir{-};
(22,6)*{};(40,6)*{} **\dir{-};
(28,12)*{};(40,12)*{} **\dir{-};
(28,18)*{};(40,18)*{} **\dir{-};
(22,0)*{};(22,6)*{} **\dir{-};
(28,0)*{};(28,18)*{} **\dir{-};
(34,0)*{};(34,18)*{} **\dir{-};
(40,0)*{};(40,18)*{} **\dir{-};
(22,0.5)*{};(40,0.5)*{} **\dir{.};
(22,1)*{};(40,1)*{} **\dir{.};
(22,1.5)*{};(40,1.5)*{} **\dir{.};
(22,2)*{};(40,2)*{} **\dir{.};
(22,2.5)*{};(40,2.5)*{} **\dir{.};
(25.2,4.5)*{_0}; (31.2,4.5)*{_0}; (37.2,4.5)*{_0};
(31.2,9)*{_1}; (37.2,9)*{_1};
(31.2,15)*{_2}; (37.2,15)*{_2}; (31,-4)*{Y_2};
(50,0)*{};(62,0)*{} **\dir{-};   
(50,3)*{};(62,3)*{} **\dir{-};
(50,6)*{};(62,6)*{} **\dir{-};
(50,12)*{};(62,12)*{} **\dir{-};
(50,18)*{};(62,18)*{} **\dir{-};
(50,24)*{};(62,24)*{} **\dir{-};
(50,27)*{};(62,27)*{} **\dir{-};
(56,30)*{};(62,30)*{} **\dir{-};
(56,36)*{};(62,36)*{} **\dir{-};
(50,0)*{};(50,27)*{} **\dir{-};
(56,0)*{};(56,36)*{} **\dir{-};
(62,0)*{};(62,36)*{} **\dir{-};
(50,0.5)*{};(62,0.5)*{} **\dir{.};
(50,1)*{};(62,1)*{} **\dir{.};
(50,1.5)*{};(62,1.5)*{} **\dir{.};
(50,2)*{};(62,2)*{} **\dir{.};
(50,2.5)*{};(62,2.5)*{} **\dir{.};
(53.2,4.5)*{_0}; (59.2,4.5)*{_0};
(53.2,9)*{_1}; (59.2,9)*{_1};
(53.2,15)*{_2}; (59.2,15)*{_2};
(53.2,21)*{_1}; (59.2,21)*{_1};
(53.2,25.5)*{_0}; (59.2,25.5)*{_0};
(59.2,28.5)*{_0};
(59.2,33)*{_1};                   (56,-4)*{Y_3};
(72,0)*{};(90,0)*{} **\dir{-};   
(72,3)*{};(90,3)*{} **\dir{-};
(72,6)*{};(90,6)*{} **\dir{-};
(78,12)*{};(90,12)*{} **\dir{-};
(78,18)*{};(90,18)*{} **\dir{-};
(78,24)*{};(90,24)*{} **\dir{-};
(78,27)*{};(90,27)*{} **\dir{-};
(72,0)*{};(72,6)*{} **\dir{-};
(78,0)*{};(78,27)*{} **\dir{-};
(84,0)*{};(84,27)*{} **\dir{-};
(90,0)*{};(90,27)*{} **\dir{-};
(72,0.5)*{};(90,0.5)*{} **\dir{.};
(72,1)*{};(90,1)*{} **\dir{.};
(72,1.5)*{};(90,1.5)*{} **\dir{.};
(72,2)*{};(90,2)*{} **\dir{.};
(72,2.5)*{};(90,2.5)*{} **\dir{.};
(75.2,4.5)*{_0}; (81.2,4.5)*{_0}; (87.2,4.5)*{_0};
(81.2,9)*{_1}; (87.2,9)*{_1};
(81.2,15)*{_2}; (87.2,15)*{_2};
(81.2,21)*{_1}; (87.2,21)*{_1};
(81.2,25.5)*{_0}; (87.2,25.5)*{_0};    (81,-4)*{Y_4};
(100,0)*{};(112,0)*{} **\dir{-};   
(100,3)*{};(112,3)*{} **\dir{-};
(100,6)*{};(112,6)*{} **\dir{-};
(106,12)*{};(112,12)*{} **\dir{-};
(106,18)*{};(112,18)*{} **\dir{-};
(106,24)*{};(112,24)*{} **\dir{-};
(106,27)*{};(112,27)*{} **\dir{-};
(106,30)*{};(112,30)*{} **\dir{-};
(100,0)*{};(100,6)*{} **\dir{-};
(106,0)*{};(106,30)*{} **\dir{-};
(112,0)*{};(112,30)*{} **\dir{-};
(100,0.5)*{};(112,0.5)*{} **\dir{.};
(100,1)*{};(112,1)*{} **\dir{.};
(100,1.5)*{};(112,1.5)*{} **\dir{.};
(100,2)*{};(112,2)*{} **\dir{.};
(100,2.5)*{};(112,2.5)*{} **\dir{.};
(103.2,4.5)*{_0}; (109.2,4.5)*{_0};
 (109.2,9)*{_1};
 (109.2,15)*{_2};
 (109.2,21)*{_1};
 (109.2,25.5)*{_0};
 (109.2,28.5)*{_0};    (106,-4)*{Y_5};
\endxy
$$
By definition, we have
\begin{align*}
\wt(Y_1) &= \Lambda_0 - 2\alpha_0, \hskip 6.7em   \wt(Y_2) = \Lambda_0 - 3\alpha_0 - 2\alpha_1 - 2\alpha_2,  \\
\wt(Y_3) & = \Lambda_0 - 5\alpha_0 - 5\alpha_1 - 2\alpha_2, \quad  \wt(Y_4) = \Lambda_0 - 5\alpha_0 - 4\alpha_1 - 2\alpha_2, \\
\wt(Y_5) &= \Lambda_0 - 4\alpha_0 - 2\alpha_1 - \alpha_2.
\end{align*}
$Y_1$ and $Y_2$ are not proper since both have two full columns with the same height. $Y_3$ is proper but it has a removable $\delta$
because a proper Young wall can be obtained by removing $\delta$-column from the shorter column of $Y_3$. $Y_4$ has two columns with the same heights
but the height is not multiple of unit length. Thus, $Y_4$ is proper and has no removable $\delta$.
$Y_5$ is also proper and it has no removable $\delta$.

Let $\mathcal{Y}(\Lambda_0)$ be the set of all proper Young walls $Y$ such that none of the columns contain a removable $\delta$. Then,
Kashiwara operators $\tilde{e}_i$ and $\tilde{f}_i$ on $\mathcal{Y}(\Lambda_0)$ can be defined by using the combinatorics of Young walls, and
$\mathcal{Y}(\Lambda_0)$ has a $U_q(\g)$-crystal structure \cite{HK02, Kang03}.
\begin{thm} \cite[Thm. 7.1]{Kang03} \label{Thm: Young walls}
The crysal $\mathcal{Y}(\Lambda_0)$ is isomorphic to the crystal $B(\Lambda_0)$ of the highest weight $U_q(\g)$-module $V_q(\Lambda_0)$.
\end{thm}

Let $\sB(\Lambda)_{\Lambda-\beta}$ be the set of isomorphism classes of irreducible  $R^\Lambda(\beta)$-modules, for $\beta \in \rlQ^+$, and define
$\sB(\Lambda)=\sqcup_{\beta\in\rlQ^+} \sB(\Lambda)_{\Lambda-\beta}$.
Then, Lauda and Vazirani defined a crystal structure on $\sB(\Lambda)$ and proved the theorem below. We note that the idea to use socle and cosocle for defining crystal structure goes back to Leclerc's interpretation of the modular branching rule as a crystal in the Hecke algebra case.

\begin{thm}[\protect{\cite[Thm.7.5]{LV11}}] \label{Thm: categorification of crystals}
The crystal $\sB(\Lambda)$ is isomorphic to the crystal $B(\Lambda)$ of the highest weight module $V_q(\Lambda)$.
\end{thm}

A realization of the crystal $B(\Lambda_0)$ in terms of partitions is easier to handle for those who are not familiar with Young walls, and we explain the realization by
$h$-restricted $h$-strict partitions from \cite[Thm.9.2]{BK01}.

For $h\in \Z_{\ge 0}$, a partition $\lambda = (\lambda_1,\lambda_2, \ldots )$ is {\it $h$-strict} if $h$ divides $\lambda_r$,
whenever $\lambda_r = \lambda_{r+1}$ for $r \ge 1$. If $h=0$, we understand it as a strict partition.
An $h$-strict partition $\lambda$ is {\it $h$-restricted} if it satisfies $\lambda_r - \lambda_{r+1} < h$ if $h | \lambda_r$ and
$\lambda_r - \lambda_{r+1} \le h$ if $h \nmid \lambda_r$. We denote by $\mathcal{RP}_h$ the set of all $h$-restricted $h$-strict
partitions. Let
\begin{align} \label{Eq: def of h}
h = 2\ell+1.
\end{align}
For a Young wall  $Y\in \mathcal{Y}(\Lambda_0)$, let $\lambda_k(Y)$ be the number of stacked blocks in the $k$th column of $Y$.
Then, it is easy to see that the map $Y \mapsto \lambda_Y= (\lambda_1(Y), \lambda_2(Y), \ldots)$ is a 1-1 correspondence between $\mathcal{Y}(\Lambda_0)$ and $\mathcal{RP}_h$.
For example, consider the case $A= A_{4}^{(2)}$, where $h = 5$.
Then, $ \lambda(Y_1) = (1,1,0, \ldots )$ and $ \lambda(Y_2) = (3,3,1,0, \ldots )$ are not $5$-strict,
$ \lambda(Y_3) = (7,5,0, \ldots ) $ is $5$-strict but not $5$-restricted, and $\lambda(Y_4) = (5,5,1,0, \ldots )$ and $\lambda(Y_5) = (6,1,0, \ldots )$
are $5$-strict and $5$-restricted.

Thus, $\mathcal{RP}_h$ has the induced crystal structure, which is easy to describe, and
the weight of $\lambda\in\mathcal{RP}_h$ is given by the residue pattern
$\nu_\delta= 01\ldots\ell\ldots10$ on the nodes of $\lambda$. We define, for $\beta=\sum_{i\in I}k_i\alpha_i \in \rlQ^+$,
$$\mathcal{RP}_h(\beta)=\{ \lambda \in \mathcal{RP}_h \mid \text{the number of $i$-nodes is $k_i$, for $i\in I.$}\}$$
Then, we have the following  proposition by
Theorem \ref{Thm: Young walls} and Theorem \ref{Thm: categorification of crystals}.
\begin{prop} \label{Prop: the number of simples}
 For $\beta \in \rlQ^+$, the number of isomorphism classes of irreducible $\KLR(\beta)$-modules is equal to $ |\mathcal{RP}_h(\beta)|$.
\end{prop}

\vskip 1em

\section{Dimension formula for $\KLR(\beta)$}

\subsection{Fock space revisited}

We need some computation in $\F$ before giving $\dim \KLR(\beta)$. As in the Fock space theory for type $A^{(1)}_{e-1}$, we may describe the action of Chevalley generators by the combinatorics of adding/removing nodes of Young diagrams, after minor modification.
Let $\lambda \vdash n$ be a shifted Young diagram.
\begin{itemize}
\item Suppose that we may remove a box of residue $i\in I$ from $\lambda$ and obtain a new shifted Young diagram. Then, we denote the resulting shifted Young diagram by
$\lambda \nearrow \sBox{i}$.
\item Similarly, if we may add a box
of residue $i\in I$ to $\lambda$ and obtain a new shifted Young diagram then we denote the resulting shifted Young diagram by $\lambda \swarrow \sBox{i} $.
\end{itemize}
We have $\wt(\lambda\nearrow \sBox{i})=\wt(\lambda)+\alpha_i$ and
$\wt(\lambda\swarrow \sBox{i})=\wt(\lambda)-\alpha_i$. Recalling the residue pattern $\nu_\delta$, a box to be removed (resp. added) has residue $i\in I$ if and only
if it is the rightmost box of a row of length $j+1$ in $\lambda$ (resp. $\lambda \swarrow \sBox{i}$) with $j\equiv i, -i-1$ modulo $h$. Thus, it follows from
$\eqref{Eq: action of efh for Binf}$ and $\eqref{Eq: Chevalley generators}$ that
\begin{align}\label{Eq: e phi and f phi}
e_i |\lambda\rangle = \sum_{\mu = \lambda \nearrow \ssBox{i}} \mathrm{m}(\lambda, \mu) |\mu\rangle, \quad
f_i |\lambda\rangle = \sum_{\mu = \lambda \swarrow \ssBox{i}} |\mu\rangle,
\end{align}
where $\mathrm{m}(\lambda, \mu) = \left\{
                           \begin{array}{ll}
                             2 & \hbox{ if $\wt(\mu) = \wt(\lambda) + \alpha_0$ and $l(\lambda) = l(\mu)$,}   \\
                             1 & \hbox{ otherwise.}
                           \end{array}
                         \right.
$
\vskip 0.5em
\begin{defn}
For $\lambda \vdash n$ and $\nu \in I^n$, we denote by
$K(\lambda,\nu)$ the number of standard tableaux of shape $\lambda$ that has the  residue sequence $\nu$. Namely, we define
\begin{align*}
K(\lambda, \nu) = | \{  T \in \ST(\lambda) \mid \nu = \res(T)  \} |.
\end{align*}
\end{defn}

\noindent
We have the following equality.
\begin{align} \label{Eq: number of ST}
|\ST(\lambda)| = \sum_{\nu \in I^{\Lambda_0-\wt(\lambda)}} K(\lambda, \nu).
\end{align}
In the next lemma,  $d$ is the scaling element.
Note that if $\wt(\lambda)=\Lambda_0-\sum_{i\in I}k_i\alpha_i$ then $k_0=-\langle d, \wt(\lambda)\rangle$.

\begin{lemma} \label{Lem: e_nu and f_nu}
Let $\lambda \vdash n$ and $\nu = (\nu_1, \ldots, \nu_n ) \in I^n$. Then we have
\begin{align*}
e_{\nu_1}e_{\nu_2} \cdots e_{\nu_n} |\lambda\rangle &= 2^{-\langle d, \wt(\lambda)\rangle - l(\lambda)}  K(\lambda, \nu) |0\rangle, \\
f_{\nu_n}f_{\nu_{n-1}} \cdots f_{\nu_1} |0\rangle &= \sum_{\mu \vdash n} K(\mu, \nu)  |\mu\rangle.
\end{align*}
\end{lemma}
\begin{proof}
The second equality is clear. We prove the first equality by induction on $n$. Since it is obvious when $n=1$, we assume $n>1$.
By definition, if $\mu = \lambda \nearrow \ssBox{i}$, then we have
$$2^{-\langle d, \wt(\lambda)\rangle-l(\lambda)} = \mathrm{m}(\lambda, \mu) 2^{-\langle d, \wt(\mu)\rangle-l(\mu)}.$$
Let $j = \nu_n$ and $\nu^- = (\nu_1, \ldots, \nu_{n-1})$.
Using the induction hypothesis and $\eqref{Eq: e phi and f phi}$, we obtain
\begin{align*}
e_{\nu_1}e_{\nu_2} \cdots e_{\nu_n} |\lambda\rangle &=
\sum_{\mu = \lambda \nearrow \ssBox{j}} \mathrm{m}(\lambda, \mu)e_{\nu_1}e_{\nu_2} \cdots e_{\nu_{n-1}} |\mu\rangle \\
& = \sum_{\mu = \lambda \nearrow \ssBox{j}} \mathrm{m}(\lambda, \mu)  2^{-\langle d, \wt(\mu)\rangle - l(\mu)}  K(\mu, \nu^-) |0\rangle  \\
&=  2^{-\langle d, \wt(\lambda)\rangle - l(\lambda)} \sum_{\mu = \lambda \nearrow \ssBox{j}}   K(\mu, \nu^-) |0\rangle \\
&=  2^{-\langle d, \wt(\lambda)\rangle - l(\lambda)}  K(\lambda, \nu) |0\rangle.
\end{align*}
Hence the desired formula follows.
\end{proof}

\subsection{Dimension formula}  \label{Section: dimension formula}

Now we are ready to state and prove the main theorem of this section. We remark that dimension of cyclotomic quiver Hecke algebras is not known
except for block algebras of cyclotomic Hecke algebras. In the case of cyclotomic Hecke algebras, 
we know $\dim R^{\Lambda}(\beta)$ from the theory of cellular algebras. For a suitable definition of $\lambda$ and $\mathsf{ST}(\lambda)$, we have
$$
\dim R^\Lambda(\beta)=\sum_{\wt(\lambda)=\Lambda-\beta} |\ST(\lambda)|^2.
$$
Theorem \ref{Thm: dimension formula} below gives the dimension of the finite quiver Hecke algebra $R^{\Lambda_0}(\beta)$ of type $A^{(2)}_{2\ell}$
in terms of shifted standard tableaux. The idea to obtain the formula is different, and we use computation in the Fock space $\F$ and the biadjointness result from \cite{KK11} and \cite{Kash11}.

We recall functors $E_i$ and $F_i$ from Section \ref{Subsection: integrable modules} and consider them in non-graded setting.
\begin{align*}
E_i &: R^\Lambda(\beta + \alpha_i)\text{-mod} \longrightarrow R^\Lambda(\beta)\text{-mod}, \\
F_i &: R^\Lambda(\beta )\text{-mod} \longrightarrow R^\Lambda(\beta+ \alpha_i)\text{-mod},
\end{align*}
Then the pair $(E_i,F_i)$ is a biadjoint pair of exact functors by \cite[Thm.4.5]{KK11}, \cite[3.2, Thm.3.5]{Kash11}. \cite[3.2]{Kash11} gives the adjunction ${\rm id}\to E_iF_i$, $F_iE_i\to{\rm id}$ on the one hand, \cite[Thm.3.5]{Kash11} gives the adjunction
$E_iF_i\to{\rm id}$, ${\rm id}\to F_iE_i$ on the other hand.

\begin{lemma}\label{symmetric hom}
We have the following equality, for $M, N\in \KLR(\beta)\text{\rm -mod}$.
$$\dim \Hom_{\KLR(\beta)}(M, N)=\dim \Hom_{\KLR(\beta)}(N, M).$$
\end{lemma}
\begin{proof}
We prove the assertion by induction on $|\beta|$. Suppose that the assertion holds for all $\beta'\in\rlQ^+$ such that $|\beta'|<|\beta|$. Then, by the induction hypothesis and the biadjointness of $E_i$ and $F_i$, we have
\begin{align*}
\dim \Hom_{\KLR(\beta)}(F_iM, N)&=\dim \Hom_{\KLR(\beta-\alpha_i)}(M, E_iN)\\
&=\dim \Hom_{\KLR(\beta-\alpha_i)}(E_iN, M)=\dim \Hom_{\KLR(\beta)}(N, F_iM).
\end{align*}
As $\KLR(\beta)$'s categorify $V(\Lambda_0)$ and
$$V(\Lambda_0)_{\Lambda_0-\beta}=\sum_{i\in I} f_iV(\Lambda_0)_{\Lambda_0-\beta+\alpha_i},$$
where $V(\Lambda_0)_\mu$ is the $\mu$-weight space, the assertion follows.
\end{proof}

\begin{thm} \label{Thm: dimension formula}
Let $\lambda \vdash n$ be a shifted Young diagram consisting of $n$ boxes. For $\beta  \in \rlQ^+$ with $|\beta|=n$ and $\nu, \nu' \in I^\beta$, we have
\begin{align*}
\dim e(\nu') \KLR(n) e(\nu) &= \sum_{\lambda \vdash n} 2^{-\langle d, \wt(\lambda)\rangle - l(\lambda)} K(\lambda, \nu')K(\lambda, \nu), \\
\dim  \KLR(\beta)  &= \sum_{\lambda \vdash n,\ \wt(\lambda)=\Lambda_0 - \beta} 2^{-\langle d, \wt(\lambda) \rangle - l(\lambda)} |\ST(\lambda)|^2, \\[5pt]
\dim  \KLR(n)  &= \sum_{\lambda \vdash n} 2^{-\langle d, \wt(\lambda)\rangle  - l(\lambda)} |\ST(\lambda)|^2.
\end{align*}
\end{thm}
\begin{proof}
We have $K_0(R^{\Lambda_0})\cong V(\Lambda_0)_\Z$
by the paragraph after Theorem \ref{Thm: categorification V}.
Let $(\ ,\ )$ be the Shapovalov form on $V(\Lambda_0)$. It is a symmetric bilinear form characterized by the following properties, where $v_0$ is a vector which generates $\{ v\in V(\Lambda_0)_\Z \mid \wt(v)=\Lambda_0\}$ as a $\Z$-module.
$$ (v_0,v_0) = 1, \quad (f_i x, y) = (x, e_i y), \text{ for } x,y \in V(\Lambda_0).$$
As $(E_i, F_i)$ is a biadjoint pair and Lemma \ref{symmetric hom} holds, the above  characterization implies that the induced Shapovalov form on $K_0(R^{\Lambda_0})$
is given as follows:
\begin{align} \label{Eq: Shapovalov}
\dim \ \Hom_{\KLR(\beta)} (M,N) = ([M] , [N]), \text{ for projective $\KLR(\beta)$-modules $M$ and $N$.}
\end{align}

We identify $U(\g)|0\rangle=V(\Lambda_0)$ and embed $V(\Lambda_0)$ to the Fock space $\F$. As $V(\Lambda_0)$ is a direct summand of $\F$ as a $\g$-module, we may extend the Shapovalov form on $V(\Lambda_0)$ to a symmetric bilinear form on $\F$ that has the property
$$ (f_i x, y) = (x, e_i y), \text{ for } x,y \in \F.$$
Further, $F_i(\KLR(n)e(\nu))=\KLR(n+1)e(\nu,i)$ and
$$F_{\nu_n} \ldots F_{\nu_1}(\KLR(0))=\KLR(n) e(\nu), \ \text{for $\nu\in I^n$.}$$
Hence, for $\nu = (\nu_1, \ldots, \nu_n), \ \nu' = (\nu_1', \ldots, \nu_n') \in I^n$,
Lemma \ref{Lem: e_nu and f_nu} and $\eqref{Eq: Shapovalov}$ imply that
\begin{align*}
\dim e(\nu') \KLR(n) e(\nu) &= \dim \Hom_{\KLR(n)}( \KLR(n) e(\nu'),\ \KLR(n) e(\nu) )\\[3pt]
&= ( f_{\nu_n'}f_{\nu_{n-1}'} \ldots f_{\nu_1'}|0\rangle,\ f_{\nu_n}f_{\nu_{n-1}} \ldots f_{\nu_1} |0\rangle ) \\[3pt]
&= \sum_{\lambda \vdash n} K(\lambda, \nu) ( f_{\nu_n'}f_{\nu_{n-1}'} \ldots f_{\nu_1'}|0\rangle,\ |\lambda\rangle ) \\[3pt]
&= \sum_{\lambda \vdash n} K(\lambda, \nu) (\, |0\rangle,\  e_{\nu_1'}e_{\nu_{2}'} \ldots e_{\nu_n'} |\lambda\rangle ) \\[3pt]
&=  \sum_{\lambda \vdash n} 2^{-\langle d, \wt(\lambda)\rangle - l(\lambda)} K(\lambda, \nu')K(\lambda, \nu).
\end{align*}
To deduce the second formula from the first, we use
$\eqref{Eq: number of ST}$ and
$$\KLR(\beta) = \bigoplus_{\nu,\nu' \in I^\beta} e(\nu') \KLR(n) e(\nu).$$
Then we have
\begin{align*}
\dim \KLR(\beta)  &= \sum_{\nu,\nu' \in I^\beta} \sum_{\lambda \vdash n} 2^{-\langle d, \wt(\lambda)\rangle - l(\lambda)} K(\lambda, \nu')K(\lambda, \nu)\\[3pt]
&= \sum_{\lambda \vdash n,\ \wt(\lambda)=\Lambda_0 - \beta}  2^{-\langle d, \wt(\lambda)\rangle - l(\lambda)}  \sum_{\nu \in I^\beta}\sum_{\nu' \in I^\beta} K(\lambda, \nu')K(\lambda, \nu) \\[3pt]
&= \sum_{\lambda \vdash n,\ \wt(\lambda)= \Lambda_0 - \beta} 2^{-\langle d, \wt(\lambda) \rangle - l(\lambda)} |\ST(\lambda)|^2.
\end{align*}
The last formula follows from
$\displaystyle \KLR(n) = \bigoplus_{\beta \in \rlQ^+,\ |\beta| = n} \KLR(\beta) $.
\end{proof}

Theorem \ref{Thm: dimension formula} has several important consequences.

\begin{cor} \label{Cor: dimension}
\begin{enumerate}
\item Let $\nu \in I^n$. Then, $e(\nu) \ne 0$ in $\KLR(n)$ if and only if $\nu$ may be obtained from a standard tableau $T$ as $ \nu = \res(T)$.
\item We have the following hook length formula, for $\beta \in \rlQ^+$.
$$ \dim \KLR(\beta) = \sum_{\lambda \vdash n, \wt(\lambda)= \Lambda_0-\beta } 2^{-\langle d, \wt(\lambda) \rangle - l(\lambda)} \left( \frac{ n! }{ \prod_{(i,j)\in \lambda} h_{i,j} } \right)^2.$$
\item For any natural number $n$, we have the following equality.
\begin{align*}
n!  = \sum_{\beta \in \rlQ^+,\ |\beta|= n} 2^{ n -\langle d, \beta\rangle } \dim \KLR(\beta).
\end{align*}
\end{enumerate}
\end{cor}
\begin{proof}
(1) It follows from Theorem \ref{Thm: dimension formula} that
 $$e(\nu) \KLR(n) e(\nu) = \sum_{\lambda \vdash n} 2^{-\langle d, \wt(\lambda)\rangle - l(\lambda)} K(\lambda, \nu)^2.$$
Then, observe that $e(\nu)\ne0$ if and only if $\dim e(\nu) \KLR(n) e(\nu)>0$.

(2) It follows from Theorem \ref{Thm: hook formula}(1) and Theorem \ref{Thm: dimension formula}.

(3) Using Theorem \ref{Thm: hook formula}(2) and Theorem \ref{Thm: dimension formula}, we obtain
\begin{align*}
n! &= \sum_{\lambda \vdash n} 2^{ n - l(\lambda)} |\ST(\lambda)|^2 \\
&= \sum_{\beta \in \rlQ^+,\ |\beta|=n} 2^{n + \langle d, \Lambda_0- \beta \rangle } \sum_{\lambda \vdash n,\ |\lambda|=\beta}
2^{-\langle d, \Lambda_0-\beta \rangle - l(\lambda)} |\ST(\lambda)|^2 \\[3pt]
&= \sum_{\beta \in \rlQ^+,\ |\beta|=n} 2^{ n - \langle d, \beta \rangle } \dim  R^{\Lambda_0}(\beta).
\end{align*}
We have proved the equality.
\end{proof}

\vskip 1em

\section{Representation type of $R^{\Lambda_0}(\beta)$} \label{Section: repn type}

\subsection{Reduction to $\KLR(k\delta)$}

We denote by $R^\Lambda(n)\text{-}\mathrm{mod}$ the category of finitely generated $R^\Lambda(n)$-modules and we define functors $E$ and $F$ by
\begin{align*}
E &:= \sum_{i\in I} E_i \ : R^\Lambda(n+1)\text{-}\mathrm{mod} \longrightarrow R^\Lambda(n)\text{-}\mathrm{mod}, \\
F &:= \sum_{i\in I} F_i \ : R^\Lambda(n)\text{-}\mathrm{mod} \longrightarrow R^\Lambda(n+1)\text{-}\mathrm{mod}.
\end{align*}
Recall that the pair $(E_i,F_i)$ is a biadjoint pair of exact functors.

\begin{lemma}
Let $\mathcal A$ and $\mathcal B$ be abelian categories, $E:{\mathcal A}\to{\mathcal B}$ an exact functor. If $F$ is right adjoint (resp. left adjoint) to $E$ then
$F$ sends an injective object (resp. projective object) to an  injective object (resp. projective object).
\end{lemma}

\begin{prop}\label{self-injectivity}
$R^{\Lambda}(\beta)$ is a self-injective algebra.
\end{prop}
\begin{proof}
As the pair $(E_i,F_i)$ is a biadjoint pair of exact functors, $F$ is right adjoint to $E$. Hence $F$ preserves injective objects.
On the other hand, by definition, $R^{\Lambda}(k+1) = F R^{\Lambda}(k)$ for $k\in \Z_{\ge0}$. Since $R^{\Lambda}(0)$ is an injective $R^{\Lambda}(0)$-module,
$R^{\Lambda}(n)$ is an injective $R^{\Lambda}(n)$-module and so are its direct summands $R^{\Lambda}(\beta)$'s.
\end{proof}

In the following, we recall several general results. Let $A$ and $B$ be $\bR$-algebras.

\begin{thm} [Krause] \label{Thm: Krause}
If $A$ and $B$ are stable equivalent, then $A$ and $B$ have the same representation type.
\end{thm}

See \cite[Thm.1]{Le00} for the Krause's theorem from \cite{Kr97}. The following theorem is well-known.

\begin{thm} [\protect{\cite[Thm.2.1]{Ri89}}] \label{Thm: Rickard}
If $A$ and $B$ are self-injective algebras such that
$D^b(A\text{-}\mathrm{mod})$ and $D^b(B\text{-}\mathrm{mod})$ are equivalent as triangulated categories, i.e.,
$A$ and $B$ are derived equivalent. Then $A$ and $B$ are stably equivalent.
\end{thm}

Hence, if two self-injective algebras are derived equivalent, they have the same number of irreducible modules, the same representation type,
isomorphic 
centers, and if they are not radical square zero Nakayama algebras then
isomorphic stable Auslander-Reiten quivers.

\begin{thm} [\protect{\cite[Thm.6.4]{CR08}}] \label{Thm: Chuang and Rouquier}
Let ${\mathcal A}$ be an artinian and noetherian $\bR$-linear abelian category
such that the endomorphism ring of any simple object is $\bR$.
Suppose that an adjoint pair $(E,F)$ of exact endo-functors on $\mathcal A$ satisfies $sl_2$-categorification axioms in \cite[5.2]{CR08}. Then
\begin{itemize}
\item[(1)] The weight space decomposition of the locally finite $sl_2(\Q)$-module $K({\mathcal A})\otimes_\Z\Q$ gives the decomposition of the category
${\mathcal A}=\oplus_{c\in\Z}\; {\mathcal A}_c.$
\item[(2)] The categories ${\mathcal A}_{-c}$ and ${\mathcal A}_c$ are derived equivalent.
\end{itemize}
\end{thm}

We consider ${\mathcal A}=\oplus_{c\in\Z} {\mathcal A}_c$, where
$${\mathcal A}_c=\bigoplus_{\mu\in\wlP:\mu(h_i)=c} R^\Lambda(\Lambda-\mu)\text{-}\mathrm{mod}.$$
Then, $(E_i, F_i)$ is an adjoint pair of exact endo-functors on  $\mathcal A$. The following is another consequence of \cite{Kash11}.

\begin{thm}\label{Kashiwara11}
For $i\in I$, the pair of functors $(E_i, F_i)$ satisfies the $sl_2$-categorification axioms.
\end{thm}
\begin{proof}
It follows from \cite[3.2]{Kash11} that we have $X_i, X_{i+1}, T\in\End(E_i^{\otimes2})$ that satisfy nil affine Hecke relations.
Thus \cite[5.3.3]{R08} implies the result.
\end{proof}

\begin{cor}
\label{cor: Chuang-Rouquier}
Let $\Lambda$ be an integral dominant weight, $V(\Lambda)$ the corresponding integrable highest weight $\g$-module.
If $ \mu \in \wlP $ appears in its weight system, that is, if the weight space $V(\Lambda)_\mu$ is nonzero, then
$R^{\Lambda}(\Lambda - \mu )$ and $R^{\Lambda}(\Lambda - w\mu )$ are derived equivalent, for $w\in\weyl$.
\end{cor}
\begin{proof}
As the weight system of an integrable module is $\weyl$-invariant, it suffices to consider the case $w=s_i$, for some $i\in I$.
Then, the result follows from Theorem \ref{Thm: Chuang and Rouquier} and Theorem \ref{Kashiwara11}.
\end{proof}

The following corollary is crucial in reducing the proof of Theorem \ref{Thm: repn type of R(beta)} to the cases when $\beta$ is a multiple of the null root $\delta$.

\begin{cor} \label{Cor: repn type}
Let $\Lambda$ be a dominant integral weight. Then, for $w\in \weyl$ and $k \in \Z_{\ge0}$, cyclotomic quiver Hecke algebras
$R^{\Lambda}(k\delta)$ and $R^{\Lambda}(\Lambda - w \Lambda + k\delta)$ have the same representation type.
\end{cor}
\begin{proof}
Since $w \delta= \delta$, Corollary \ref{cor: Chuang-Rouquier} tells that
$R^{\Lambda}(k\delta)$ and $R^{\Lambda}(\Lambda - w \Lambda + k\delta)$ are derived equivalent. Then,
Proposition \ref{self-injectivity} and Theorem \ref{Thm: Rickard} imply that they are stably equivalent.
Thus, the assertion follows from Theorem \ref{Thm: Krause}.
\end{proof}

In the rest of the paper, we focus on finite quiver Hecke algebras $\KLR(\beta)$ of type $A^{(2)}_{2\ell}$. Then, every nonzero
$\KLR(\beta)$ has the form $\KLR(\Lambda_0-w\Lambda_0+k\delta)$, which has the same representation type as
$\KLR(k\delta)$ by Corollary \ref{Cor: repn type}.

\vskip 1em

We recall the following theorem. Let $\Rad \, A$ be the two-sided ideal generated by paths of positive length, for a path algebra $A$.

\begin{thm}[\protect{\cite[Thm.1]{GR79}}]\label{Brauer tree algebra}
Brauer tree algebras are stably equivalent to symmetric Nakayama algebras $A/\Rad^{lm+1}A$, for the path algebra $A$ of a cyclic quiver of length $l$
and $m\in\Z_{\geq0}$.
\end{thm}

Thus, Brauer tree algebras have finite representation type by Theorem \ref{Thm: Krause}.
Of course, the fact had been known before Theorem \ref{Brauer tree algebra} appeared, but we do not know a good reference for this well-known fact.
According to Professor Asashiba,
arguments in Janusz \cite{Ja69} may be read as another proof, which is more direct and elementary. We are grateful to him for the explanation.

\vskip 1em

\subsection{The algebra $R^{\Lambda_0}(\delta)$}
In this subsection, we prove that $\KLR(\delta)$ is a Brauer tree algebra. Hence, it has finite representation type by Theorem \ref{Brauer tree algebra}.
We assume that $Q_{i,j}(u,v)=1$ if $a_{ij} = 0$ for simplifying computations. This assumption is not essential. Recall that $h=2\ell+1$.

For $0\le i\le\ell-1$, we consider two row partitions $ \lambda(i) = (h-i-1, i) \vdash h-1$.  Note that
$\wt(\lambda(i)) = \Lambda_0-\delta+\alpha_i$ and the residues are given as follows:
\begin{align} \label{Eq: residue of lambda i}
\xy
(0,6)*{};(60,6)*{} **\dir{-};
(0,0)*{};(60,0)*{} **\dir{-};
(6,-6)*{};(28,-6)*{} **\dir{-};
(0,0)*{};(0,6)*{} **\dir{-};
(6,-6)*{};(6,6)*{} **\dir{-};
(12,-6)*{};(12,6)*{} **\dir{-};
(22,-6)*{};(22,6)*{} **\dir{-};
(28,-6)*{};(28,6)*{} **\dir{-};
(38,0)*{};(38,6)*{} **\dir{-};
(44,0)*{};(44,6)*{} **\dir{-};
(54,0)*{};(54,6)*{} **\dir{-};
(60,0)*{};(60,6)*{} **\dir{-};
(3.2,3)*{_0}; (9.2,3)*{_1}; (17,3)*{\cdots}; (25.2,3)*{_i}; (33,3)*{\cdots}; (41.2,3)*{_\ell}; (49,3)*{\cdots}; (57.5,3)*{_{i+1}};
 (9.2,-3)*{_0}; (17,-3)*{\cdots}; (25.5,-3)*{_{i-1}};
\endxy
\end{align}

 We set
$$L_i = \bigoplus_{ T \in \ST(\lambda(i)) }  \bR T  $$
and define an $\KLR(\delta - \alpha_i)$-module structure on $L_i$ as follows: for $\nu \in I^{\delta - \alpha_i},\
1\le k\le h-1$ and $1\le l<h-1,$
\begin{align} \label{Eq: def of Li}
e(\nu) T = \left\{
             \begin{array}{ll}
               T & \hbox{ if } \nu = \res(T), \\
               0 & \hbox{ otherwise,}
             \end{array}
           \right. \
x_k T = 0, \ \ \
\psi_l T = \left\{
             \begin{array}{ll}
               s_l T & \hbox{ if $s_l T$ is standard,}  \\
               0 & \hbox{ otherwise,}
             \end{array}
           \right.
\end{align}
where $s_l T$ is the tableau obtained from $T$ by exchanging the entries $l$ and $l+1$.
We check that it is well-defined. It is easy to see from the residue pattern $\eqref{Eq: residue of lambda i}$
that $\nu_k=\nu_{k+1}$ does not occur, for $\nu=\res(T)$ with $T\in \ST(\lambda(i))$.
Thus, it is straightforward to check the defining relations except those for $\psi_k^2$ and for
$\psi_{k+1}\psi_k\psi_{k+1}-\psi_k\psi_{k+1}\psi_k$.

Let $T \in \ST(\lambda(i))$ and write $\res(T) = (\nu_1, \ldots, \nu_{h-1})$.
Since $0\le i \le \ell-1$, one can also show from the residue pattern $\eqref{Eq: residue of lambda i}$ that
\begin{enumerate}
\item[(i)]Let $(1,j)\in\lambda(i)$ and $(2,j')\in\lambda(i)$ be such that $ j > j' $. Then $a_{pq}=0$, for $p=\res(1,j)$ and $q=\res(2,j')$.
\item[(ii)] If $\nu_{k} = \nu_{k+2}$, then $(\nu_{k}, \nu_{k+1}, \nu_{k+2}) = (0,1,0)$ or $(\ell-1,\ell,\ell-1)$.
\end{enumerate}

We show that $\psi_k^2e(\nu)=e(\nu)$ if $a_{\nu_k,\nu_{k+1}}=0$, and $\psi_k^2e(\nu)=0$ otherwise.
Let $\nu=\res(T)$, for $T\in \ST(\lambda(i))$. Since
(i) implies that $s_k T$ is standard if and only if $a_{\nu_k, \nu_{k+1}} =0$, we have
$$ \psi_k^2 T = \left\{
                                        \begin{array}{ll}
                                          T & \hbox{ if  $a_{\nu_k,\nu_{k+1}}=0$,} \\
                                          0 & \hbox{ otherwise.}
                                        \end{array}
                                      \right.
 $$
Thus, we have proved that $\psi_k^2e(\nu)=Q_{\nu_k,\nu_{k+1}}(x_k,x_{k+1})e(\nu)$ holds on $L_i$.

As $s_kT, s_{k+1}s_kT, s_ks_{k+1}s_kT$ are all standard if and only if $s_{k+1}T, s_ks_{k+1}T, s_{k+1}s_ks_{k+1}T$ are all standard, the operators
$\psi_k$ and $\psi_{k+1}$ on $L_i$ satisfy the Artin braid relation. Thus, we show that the terms with $\nu_k=\nu_{k+2}$ in the defining relation vanish.
Since (ii) implies that we only have to consider $(\nu_{k}, \nu_{k+1}, \nu_{k+2}) = (0,1,0)$ or $(\ell-1,\ell,\ell-1)$, and
$\deg(\psi_k\psi_{k+1}\psi_{k} e(\nu)) > 0 $ in both cases, they vanish as desired.

\begin{lemma} \label{Lem: Li} For $i = 0,1,\ldots, \ell-1$,
\begin{enumerate}
\item  $\KLR(\delta-\alpha_i)$ is a simple algebra,
\item $L_i$ is an irreducible $\KLR(\delta-\alpha_i)$-module of dimension ${ h-2\choose i} - { h-2\choose i-1}$.
\end{enumerate}
\end{lemma}
\begin{proof}
(1) Since $\Lambda_0 - \delta + \alpha_i$ is a maximal weight of $V(\Lambda_0)$,
$$ \dim V(\Lambda_0)_{\Lambda_0 - \delta + \alpha_i} = 1,\quad  \Lambda_0 - \delta + \alpha_i =  w \Lambda_0, $$
for some $w\in \weyl$ \cite[Lem.12.6]{Kac90}. Then, as $R^{\Lambda_0}(0)$ is a simple algebra, the assertion follows from  Corollary \ref{Cor: repn type} and the fact that derived equivalence preserves the center.

(2) Theorem \ref{Thm: hook formula} gives the dimension of $L_i$ as follows:
\begin{align*}
\dim L_i &= \frac{(h-1)!}{ (h-1) \left( \frac{(h-i-1)!}{h-2i-1} \right) i!   }
= \frac{(h-2)!}{ (h-i-2)!i! } - \frac{(h-2)!}{ (h-i-1)!(i-1)! }.
\end{align*}
Next we show that $\dim e(\nu)L_i\le 1$, for any $\nu$. Then, a standard argument shows that $L_i$ is irreducible. But, if one tries to enumerate all possible standard tableaux $T$ with $\res(T)=\nu$, one finds that it is unique if it exists. It follows that $\dim e(\nu)L_i\le 1$.
\end{proof}

We now extend the $R^{\Lambda_0}(\delta-\alpha_i)$-module $L_i$ to an $R^{\Lambda_0}(\delta)$-module $S_i$ as follows.
As the generators $x_k$, for $1\le k\le h-1$, and
$\psi_l$, for $1\le l<h-1$, act as \eqref{Eq: def of Li}, we define the action of $x_h$ and $\psi_{h-1}$. We declare that both act as $0$.
The idempotents  $e(\nu)$, for $\nu \in I^h$, act as
$$ e(\nu) T = \left\{
                \begin{array}{ll}
                  T & \hbox{ if } \nu = \res(T) * i, \\
                  0 & \hbox{ otherwise.}
                \end{array}
              \right.
 $$
Here, $\res(T) * i$ is the sequence obtained from $\res(T)$ by adding $i$ at the right end. We check that it is well-defined.
Let $\nu=\res(T) * i$. Then, it follows from $\eqref{Eq: residue of lambda i}$ that
$$ \nu_{h-2} \ne i \quad  \text{ and } \quad \nu_{h-1} = i-1 \text{ or } i+1,$$
if $\ell\ge2$, $\deg Q_{\nu_{h-2},\nu_{h-1}}\ge2$ and $\nu_{h-1} = i\pm1$ if $\ell=1$.
Then, $\nu_{h-1}=i\pm1$ implies that $\nu_{h-1}=\nu_h$ does not occur. Thus, the relations which involve $x_h$ all hold.
It also follows that $a_{\nu_{h-1}\nu_h}=0$ does not occur, so that the relation for $\psi_{h-1}^2$ holds.
Finally, $\nu_{h-2}\ne\nu_h$ if $\ell\ge2$ and $\deg Q_{\nu_{h-2},\nu_{h-1}}\ge2$ if $\ell=1$ implies that the relation for $\psi_{h-1}\psi_{h-2}\psi_{h-1}-\psi_{h-2}\psi_{h-1}\psi_{h-2}$.
Note that $S_i$ is a homogeneous representation in the sense of Kleshchev and Ram \cite{KR08}, and
\begin{align} \label{Eq: E of S}
E_j S_i =    \left\{
                \begin{array}{ll}
                  L_i & \hbox{ if } j = i, \\
                  0 & \hbox{ if } j \ne i.
                \end{array}
              \right.
\end{align}

\begin{lemma}
The set $\{ S_i \mid 0\le i\le \ell-1 \}$ is a complete set of irreducible $R^{\Lambda_0}(\delta)$-modules.
\end{lemma}
\begin{proof}
By Proposition \ref{Prop: the number of simples}, the number of pairwise non-isomorphic irreducible $R^{\Lambda_0}(\delta)$-modules is $\ell$.
On the other hand, Lemma \ref{Lem: Li} (2) and $\eqref{Eq: E of S}$ tell us that $S_i$'s are irreducible and pairwise non-isomorphic. Thus, we have the assertion.
\end{proof}

We now consider the module $F_i L_i$. Note that $F_i L_i$ is a projective-injective module since the functor $F_i$ preserves projectivity and injectivity.
By the biadjointness of $E_i$ and $F_i$,
\begin{equation} \label{Eq: Hom od FL}
\begin{aligned}
\Hom(S_j, F_iL_i) &\simeq
\Hom(E_i S_j, L_i) \simeq \left\{
\begin{array}{ll}
 \bR & \hbox{ if } j=i, \\
 0 & \hbox{ if } j\ne i,
\end{array}
\right.\\
\Hom(F_iL_i, S_j) &\simeq
\Hom( L_i, E_i S_j) \simeq \left\{
\begin{array}{ll}
 \bR & \hbox{ if } j=i, \\
 0 & \hbox{ if } j\ne i,
\end{array}
\right.
\end{aligned}
\end{equation}
which implies that $F_iL_i$ is indecomposable and $\Top(F_iL_i) = \Soc(F_iL_i) = S_i$.
So, $F_iL_i$ is the projective cover of $S_i$.

\begin{thm} If $\ell\ge2$, the radical series of $F_i L_i$, for $0\le i\le \ell-1$, is given as follows:
\begin{align*}
F_0 L_0 \simeq \begin{array}{c}  S_0 \\  S_0 \oplus S_1 \\  S_0 \end{array}, \ \ \ \
F_i L_i \simeq \begin{array}{c}  S_i \\  S_{i-1} \oplus S_{i+1} \\  S_i \end{array}  \;(i\ne 0,\ell-1),\ \ \ \
F_{\ell-1} L_{\ell-1} \simeq \begin{array}{c}  S_{\ell-1} \\  S_{\ell-2} \\  S_{\ell-1} \end{array}.
\end{align*}
If $\ell=1$, $S_0$ is the unique irreducible module and $F_0L_0$ is a uniserial module of length $3$.
\end{thm}
\begin{proof}
By definition, $E_j L_i = 0$, for $j \ne i\pm 1$, and we know that
\begin{align*}
E_j L_i \simeq E_i L_j , \text{ for $j=i\pm 1$ }
\end{align*}
by comparing their characters, since both are irreducible. For $i \ne j$, by the biadjointness of $E_j$ and $F_j$ and
\cite[Thm.5.1]{KK11}, we have
\begin{align*}
\Hom(F_iL_i, F_jL_j) \simeq \Hom(E_jF_iL_i, L_j) \simeq \Hom(F_iE_jL_i, L_j) \simeq \Hom(E_jL_i, E_iL_j),
\end{align*}
which gives
$$ \dim \Hom(F_iL_i, F_jL_j) = \left\{
                                 \begin{array}{ll}
                                   1 & \hbox{ if } j= i\pm 1,  \\
                                   0 & \hbox{ if } j \ne i, i\pm 1.
                                 \end{array}
                               \right.
 $$

We now consider the case $i=j$.
Since $E_i L_i = 0$ and $\langle h_i, \Lambda_0 - \delta + \alpha_i \rangle > 0$, we have, by Theorem \ref{Thm: categorification},
$$ E_i F_i L_i \simeq  L_i^{\oplus \langle h_i, \Lambda_0 - \delta + \alpha_i \rangle} $$
So, we have
\begin{align*}
\Hom(F_iL_i, F_iL_i) \simeq \Hom(L_i, E_iF_iL_i) \simeq \bR^{\oplus \langle h_i, \Lambda_0 - \delta + \alpha_i \rangle} .
\end{align*}
To summarize, we have the following decomposition numbers.
$$ [F_iL_i: S_j] = \left\{
                     \begin{array}{ll}
                       0 & \hbox{ if } j \ne i \pm 1, \\
                       1 & \hbox{ if } j = i \pm 1, \\
                       2 & \hbox{ if } j=i\ne 0, \\
                       3 & \hbox{ if } j=i= 0.
                     \end{array}
                   \right.
 $$
In other words, we have
\begin{align*}
[F_0 L_0] &=\begin{cases} 3[S_0] + [S_1] \ &\text{if $\ell\ge2$}, \\ 3[S_0] \ &\text{if $\ell=1$},\end{cases}\\
[F_i L_i] &= 2[S_i] + [S_{i-1}] + [S_{i+1}] \ \ (1\le i\le \ell-2),  \\
[F_{\ell-1} L_{\ell-1}] &= 2[S_{\ell-1}] + [S_{\ell-2}] \ \text{if $\ell\ge2$},
\end{align*}
in the Grothendieck group $K_0(\KLR(\delta)\text{\rm -mod})$.
Since $\Top(F_iL_i) = \Soc(F_iL_i) = S_i$, we only have to show that the heart of $F_iL_i$ is a direct sum of two irreducible modules when $i\ne\ell-1$.
Recall that the anti-involution of $\KLR(\delta)$, which fixes the generators elementwise, defines a left module structure on the $\bR$-dual of any module.
Further, the character consideration shows that irreducible modules are self-dual. This implies that the $\bR$-dual of $F_iL_i$ is isomorphic to
$F_iL_i$ itself. It follows that the heart of $F_iL_i$ is self-dual, and the assertion follows.
\end{proof}

As a corollary, we have the following theorem, for $\ell\ge1$. (1) is by definition of the Brauer tree algebra, and  (2) follows from Theorem \ref{Brauer tree algebra}.

\begin{thm} \label{k=1}
\begin{enumerate}
\item $\KLR(\delta)$ is the Brauer tree algebra associated with Brauer graph
\vskip 0.3em
$$
\xy
(2,0)*{};(14,0)*{} **\dir{-};
(18,0)*{};(30,0)*{} **\dir{-};
(34,0)*{};(40,0)*{} **\dir{-};
(41,0)*{};(48,0)*{} **\dir{.};
(49,0)*{};(56,0)*{} **\dir{-};
(60,0)*{};(72,0)*{} **\dir{-};
(76,0)*{};(88,0)*{} **\dir{-};
(0,-0.25)*{\bullet};
(0,-5)*{e=2};
(16,-0.25)*{\circ};
(32,-0.25)*{\circ};
(58,-0.25)*{\circ};
(74,-0.25)*{\circ};
(90,-0.25)*{\circ};
(9,3)*{\text{ \small{$S_0$} } };
(25,3)*{\text{ \small{$S_1$} } };
(66,3)*{\text{ \small{$S_{\ell-2}$} } };
(82,3)*{\text{ \small{$S_{\ell-1}$} } };
\endxy
\  .
$$
\vskip 0.5em
\item $\KLR(\delta)$ is not semisimple and has finite representation type.
\end{enumerate}
\end{thm}

\vskip 1em

\subsection{ The algebra $R^{\Lambda_0}(2\delta)$}

We define $\nu^k\in I^{h+k}$, for $1\le k\le h$, by adding residues of $\nu_\delta$ one by one. Namely,
\begin{gather*}
\nu^1=\nu_\delta*(0), \ \nu^2=\nu_\delta*(01), \ \ldots \ , \ \nu^{\ell+1}=\nu_\delta*(01\cdots\ell), \ \ldots \\
\ldots \ , \ \nu^{h-1}=\nu_\delta*(012\cdots\ell\cdots 21), \ \nu^h=\nu_\delta*\nu_\delta.
\end{gather*}
We denote the corresponding idempotent by $e_k=e(\nu^k)$. Then,
$e_k\in \KLR(\beta_k)$, where $\beta_k$ are
\begin{gather*}
\beta_1=\delta+\alpha_0, \ \beta_2=\delta+\alpha_0+\alpha_1, \ \ldots \ , \
\beta_{\ell+1}=\delta+\alpha_0+\cdots+\alpha_\ell, \ \ldots \\
\ldots \ , \ \beta_{h-1}=2\delta-\alpha_0, \ \beta_h=2\delta.
\end{gather*}
We define $i_k\in I$, for $1\le k<h$, by $\alpha_{i_k}=\beta_{k+1}-\beta_k$.

\vskip 1em

\begin{lemma}\label{iteration of spherical subalgebra}
\begin{itemize}
\item[(1)] $\dim e_k\KLR(\beta_k)e_k = 12$, for $1\le k\le h-1$.
\item[(2)] We have an isomorphism of $\bR$-algebras
$$e_k\KLR(\beta_k)e_k \cong e_{k+1}\KLR(\beta_{k+1})e_{k+1}, \ \text{for  $1\le k<h-1$}, $$
which sends
$x_he_k, \ x_{h+1}e_k, \ \psi_he_k$ to $x_he_{k+1}, \ x_{h+1}e_{k+1}, \ \psi_he_{k+1}$, respectively.
\end{itemize}
\end{lemma}
\begin{proof}
(1) We use Theorem \ref{Thm: dimension formula}. In counting standard tableaux $T$ with $\res(T)=\nu^k$, one finds that the filling of $1,\dots,h-1$ are unique, the filling of $h, h+1$ is either $(1,h), (1,h+1)$ or $(1,h), (2,2)$, and after that, there is unique way to fill in $h+2,\dots,h+k$. Thus,
$$K(\lambda,\nu_k)=\begin{cases} 1 \quad& \text{if $\lambda=(h+k)$}, \\
                                              2 \quad& \text{if $\lambda=(h,k)$}, \\
                                              0 \quad& \text{otherwise.}\end{cases}$$
It follows that $\dim e_k\KLR(\beta_k)e_k = 2^2\cdot 1^2+2^1\cdot 2^2=12$.

\bigskip
\noindent
(2)  We can easily find that
\begin{gather*}
\langle h_0, \ \Lambda_0-\delta\rangle, \ \langle h_1, \ \Lambda_0-\delta-\alpha_0\rangle,
\ldots, \ \langle h_\ell, \ \Lambda_0-\delta-\alpha_0-\cdots-\alpha_{\ell-1}\rangle, \\
\qquad \langle h_{\ell-1}, \ \Lambda_0-2\delta+\alpha_0+\cdots+\alpha_{\ell-1}\rangle,
 \ \ldots \ , \ \langle h_1, \ \Lambda_0-2\delta+\alpha_0+\alpha_1\rangle,
\end{gather*}
are all positive, for $\ell\ge1$. Thus, Theorem \ref{Thm: categorification} implies that
there is a $(\KLR(\beta_k), \KLR(\beta_k))$-bimodule monomorphism
$$ \KLR(\beta_k) \rightarrow e(\beta_k,i_k)\KLR(\beta_{k+1})e(\beta_k,i_k), $$
for $1\le k< h-1$, which respects generators. Thus, we have an $(e_k\KLR(\beta_k)e_k, e_k\KLR(\beta_k)e_k)$-bimodule monomorphism
$$ e_k\KLR(\beta_k)e_k \rightarrow e_{k+1}\KLR(\beta_{k+1})e_{k+1}, $$
for $1\le k< h-1$, such that $x_he_k, x_{h+1}e_k, \psi_he_k \mapsto x_he_{k+1}, x_{h+1}e_{k+1}, \psi_he_{k+1}$, respectively. As both algebras are $12$-dimensional, it is a $\bR$-algebra isomorphism.
\end{proof}

\begin{defn}
The cyclotomic nilHecke algebra $NH^m_2$, for a positive integer $m$, is the $\bR$-algebra defined by generators $y_1, y_2, \psi$ and the relations
$$y_1^m=0, \ y_2\psi-\psi y_1=1=\psi y_2-y_1\psi, \ \psi^2=0.$$
\end{defn}
We need the fact
$NH^3_2 \cong \mathrm{Mat}(2, \bR[x]/(x^3))$ \cite[(5.4)]{Lauda11} in the proof of next
proposition.

\vskip 1em

\begin{lemma}\label{x_{h-1}s_{h-1}e_1}
We have $x_{h-1}e(s_{h-1}\nu^1)=0.$
\end{lemma}
\begin{proof}
If $\ell=1$, then $e(s_2\nu^1)=0$, so that there is nothing to prove. Suppose that $\ell\ge2$ and
$Q_{i,j}(u,v)=1$ if $a_{ij} = 0$ and
$$\mathcal{Q}_{i,j} (u,v) = \left\{
                              \begin{array}{ll}
                                u+v & \hbox{ if $a_{ij} = a_{ji} = -1$}, \\
                                u^2 + v & \hbox{ if $a_{ij} = -2$ and $a_{ji} = -1$},
                              \end{array}
                            \right.
$$
for simplifying computation. This assumption is not essential.

We define $\mu$ by $$\mu=s_3s_4\cdots s_{h-1}\nu^1=(0,1,0,2,3,\ldots,\ell-1,\ell,\ell-1,\ldots,1,0).$$
By Corollary \ref{Cor: dimension}, and the defining relations, we have
\begin{itemize}
\item[(i)] $e(s_1\mu)=e(s_2\mu)=0,$
\item[(ii)] For $3\le r\le h-2$, we have
$$\psi_r\psi_{r-1}\cdots\psi_3e(\mu)\psi_3\cdots\psi_{r-1}\psi_r=e(s_{r+1}s_{r+2}\cdots s_{h-1}\nu^1).$$
\end{itemize}
Then, it follows that
$$x_3e(\mu)=(x_1+x_3)e(\mu)=(\psi_2\psi_1\psi_2-\psi_1\psi_2\psi_1)e(\mu)=0,$$
which yields
\begin{align*}
x_{h-1}e(s_{h-1}\nu^1)&= x_{h-1}\left(\psi_{h-2}\cdots \psi_3e(\mu)\psi_3\cdots \psi_{h-2}\right)\\
&=\psi_{h-2}\cdots \psi_3x_3e(\mu)\psi_3\cdots \psi_{h-2}\\
&=0.
\end{align*}
We have proved the result.
\end{proof}

\begin{prop} \label{spherical subalgebra}
\begin{itemize}
\item[(1)]
The following set spans $e_1\KLR(\beta_1)e_1$ as a $\bR$-vector space.
$$\left\{ x_h^ax_{h+1}^be_1, \ x_h^ax_{h+1}^b\psi_he_1 \mid a, b\in\Z_{\ge0}\right\}.$$
\item[(2)]
We have the following isomorphism of $\bR$-algebras.
$$NH^3_2 \cong e_1\KLR(\beta_1)e_1: \ y_1\mapsto x_he_1, \ y_2\mapsto x_{h+1}e_1, \ \psi\mapsto \psi_he_1.$$
\item[(3)]
The algebra $e_{h-1}\KLR(\beta_{h-1})e_{h-1}$ is generated by $x_he_{h-1}, x_{h+1}e_{h-1}, \psi_he_{h-1}$, which obey the affine nilHecke relations, and we have the following isomorphism of $\bR$-algebras.
$$ e_{h-1}\KLR(\beta_{h-1})e_{h-1}\cong \mathrm{Mat}(2, \bR[x]/(x^3)).$$
\end{itemize}
\end{prop}
\begin{proof}
(1) In the proof, we assume that $Q_{i,j}(u,v)=1$ if $a_{ij} = 0$ and
$$\mathcal{Q}_{i,j} (u,v) = \left\{
                              \begin{array}{ll}
                                u+v & \hbox{ if $a_{ij} = a_{ji} = -1$}, \\
                                u^2 + v & \hbox{ if $a_{ij} = -2$ and $a_{ji} = -1$},\\
                                u^4 + v & \hbox{ if $a_{ij} = -4$ and $a_{ji} = -1$},
                              \end{array}
                            \right.
$$
for simplifying computation. This assumption is not essential. Let $e=e_1$ and $\nu=\nu^1$.
Observe that $s_i \nu$, for $1\le i\le h-2$, can not be the residue sequence of a standard tableau. Thus,
$e\psi_i=0=\psi_ie$, for $1\le i\le h-2$, by Corollary \ref{Cor: dimension}. Suppose that $\ell\ge2$. We have
$$e\psi_1^2 e=(x_1^2+x_2)e, \;\; e\psi_i^2e=(x_i+x_{i+1})e, \ \text{for $2\le i\le \ell-1$}, \;\; e\psi_\ell^2 e=(x_\ell^2+x_{\ell+1})e,$$
so that, starting with $x_1=0$, we obtain $x_ie=0$, for $1\le i\le\ell+1$. $x_{\ell+2}e=0$ follows from
$$e\psi_{\ell+1}\psi_\ell\psi_{\ell+1}e-e\psi_\ell\psi_{\ell+1}\psi_\ell e=(x_\ell+x_{\ell+2})e.$$
Then, $e\psi_i^2e=(x_i+x_{i+1})e$, for $\ell+2\le i\le h-2$, proves that $x_ie=0$, for $1\le i\le h-1$.
If $\ell=1$, then $x_1=0$ and $e\psi_1^2 e=(x_1^4+x_2)e$ proves $x_ie=0$, for $1\le i\le h-1$.

For each $w\in S_n$, we fix a reduced expression and defines $\psi_w$. Then, $e_1\KLR(\beta_1)e_1$ is spanned by the set
$$\left\{ x_h^ax_{h+1}^be\psi_we \mid a,b\in\Z_{\ge0}, \ w\in S_{h+1} \right\}.$$
Hence, it is enough to show that $e\psi_we=0$ unless $w=1$ or $w=s_h$. Recall the distinguished coset representatives for
$S_{n-1}\backslash S_n$:
$$S_n=\bigsqcup_{i=1}^n S_{n-1}s_{n-1}\cdots s_i.$$
Hence, we may choose the reduced expression as $w=w_2\cdots w_h$, where
\begin{align*}
w_h&\in \{ s_hs_{h-1}\cdots s_1, \ s_hs_{h-1}\cdots s_2, \ \ldots, \ s_h, \ 1\} \\
w_{h-1}&\in \{ s_{h-1}s_{h-2}\cdots s_1, \ s_{h-1}s_{h-2}\cdots s_2, \ \ldots, \ s_{h-1}, \ 1\} \\
&\quad\ldots\ldots \\
w_2&\in \{ s_1, \ 1\}.
\end{align*}
Suppose that $e\psi_we\ne0$. If $w_2\ne 1$, $e\psi_1=0$ implies $e\psi_we=0$, so that $w_2=1$. Arguing similarly,
we obtain $w_2=\cdots=w_{h-2}=1$. Using $\psi_ie=0$, for $1\le i\le h-2$, we deduce that
$$w_h\in\{s_hs_{h-1}, \ s_h, \ 1\}.$$
But, $s_hs_{h-1}\nu=01\cdots\ell\cdots2001$ if $\ell\ge2$, and $s_hs_{h-1}\nu=0001$ if $\ell=1$. In either case, it cannot be of the form $\res(T)$, for a standard tableau $T$.
Hence, we have either $w_h=1$ or $w_h=s_h$. It follows that $e\psi_we=e\psi_{w_{h-1}}e\psi_{w_h}$, and using  $\psi_ie=0$, for $1\le i\le h-2$ again,
we have either $w_{h-1}=1$ or $w_{h-1}=s_{h-1}$. But $w_{h-1}=s_{h-1}$ implies $e\psi_we=e(\nu)e(s_{h-1}\nu)\psi_{h-1}\psi_{w_h}=0$ because
$s_{h-1}\nu\ne\nu$, and we must have $w_{h-1}=1$. We have proved (1).

\medskip
\noindent
(2) If $\ell\ge2$, then $e_1\psi_{h-1}^2e_1=(x_{h-1}+x_h^2)e_1=x_h^2e_1$ and Lemma
\ref{x_{h-1}s_{h-1}e_1} implies
$$x_h^3e_1=e_1\psi_{h-1}^2e_1x_h=e_1\psi_{h-1}x_{h-1}e(s_{h-1}\nu^1)\psi_{h-1}e_1=0.$$
If $\ell=1$, then $e_1\psi_i=0=\psi_ie_1$, for $i=1,2$, implies
$$x_3^3e_1=(x_1^3+x_1^2x_3+x_1x_3^2+x_3^3)e_1=
e_1(\psi_2\psi_1\psi_2-\psi_1\psi_2\psi_1)e_1=0.$$
Thus, $x_h^3e_1=0$, for all $\ell\ge1$, and we may define a homomorphism of $\bR$-algebras
$$NH^3_2 \rightarrow e_1\KLR(\beta_1)e_1 : \
y_1\mapsto x_he_1, \ y_2\mapsto x_{h+1}e_1, \ \psi\mapsto \psi_he_1,$$
which is an epimorphism by (1). As both algebras are $12$ dimensional, it is an isomorphism.

\noindent
(3) follows from (2) and Lemma \ref{iteration of spherical subalgebra}(2).
\end{proof}

\begin{prop}\label{structure of sph. subalgebra}
The algebra $e_h\KLR(2\delta)e_h$ is a quotient algebra of
$\mathrm{Mat}(2,\bR[x]/(x^3)) \otimes_\bR \bR[t]  $ by the ideal generated by an element of the form $ t^3 - c_2 t^2 - c_1 t - c_0 $, for some
$$c_0, c_1, c_2 \in \mathrm{Mat}(2,\bR[x]/(x^3)).$$
\end{prop}
\begin{proof}
By explicit enumeration of possible standard tableaux, Theorem \ref{Thm: dimension formula} computes
$$ \dim e_h\KLR(2\delta)e_h = 36. $$

We consider the algebra homomorphism
\begin{align*}
e_{h-1}\KLR(\beta_{h-1})e_{h-1}  \otimes_\bR \bR[t] \longrightarrow
e_h\KLR(2\delta)e_h
\end{align*}
defined by $ m \otimes t^k \mapsto m x_{2h}^k e_h$, for
$m\in e_{h-1}\KLR(\beta_{h-1})e_{h-1}$ and $k\in\Z_{\ge0}$.

This algebra homomorphism is well-defined because $x_{2h}$ commutes with $x_he_h, x_{h+1}e_h, \psi_he_h$, and
$x_he_{h-1}, x_{h+1}e_{h-1}, \psi_he_{h-1}$ generate
$e_{h-1}\KLR(\beta_{h-1})e_{h-1}$ by Proposition \ref{spherical subalgebra}(3).

As $\langle h_0, \Lambda_0 -2\delta+\alpha_0 \rangle = 3$,
Theorem \ref{Thm: categorification} implies that its restriction to
$$\bigoplus_{k=0,1,2} e_{h-1}\KLR(\beta_{h-1})e_{h-1}\otimes t^k \longrightarrow
e_h\KLR(2\delta)e_h$$
is an isomorphism of $\bR$-vector spaces. In fact, the isomorphism of functors in
Theorem \ref{Thm: categorification} is given in an explicit manner, and the terms in the direct sum are given by multiplication by $x_{2h}^k$, for $0\le k\le 2$. See the statement of \cite[Thm.3.4]{Kash11}, for example.

Thus, $t^3=c_2t^2+c_1t+c_0$, for some $c_0, c_1, c_2\in e_{h-1}\KLR(\beta_{h-1})e_{h-1}$, and if we consider the factor algebra by the cubic relation, the algebra homomorphism induces an isomorphism to $e_h\KLR(2\delta)e_h$. Recalling the isomorphism of $\bR$-algebras
$$e_{h-1}\KLR(\beta_{h-1})e_{h-1}\cong \mathrm{Mat}(2, \bR[x]/(x^3))$$
from Proposition \ref{spherical subalgebra}(3), we have the result.
\end{proof}

\begin{prop}\label{k=2}
The algebra $R^{\Lambda_0}(2\delta)$ is wild, for all $\ell\ge1$.
\end{prop}
\begin{proof}
Let $A=e_h\KLR(2\delta)e_h$. It suffices to show that $A$ is wild.
We define a two-sided ideal of $\mathrm{Mat}(2,\bR[x]/(x^3)) \otimes \bR[t]$ by
$$J=\mathrm{Mat}(2,\bR[x]/(x^3)) \otimes t\bR[t] +\mathrm{Mat}(2,x\bR[x]/(x^3)) \otimes \bR[t].$$
We know from Proposition \ref{structure of sph. subalgebra} that
$A/\Rad^3 A$ is isomorphic to
$$\mathrm{Mat}(2,\bR[x]/(x^3)) \otimes \bR[t]/J^3$$
as $\bR$-algebras. We show that this algebra is wild. Then it follows that $A$ is wild. Let $E_{11}$ be the matrix unit. Then, for the idempotent $f=E_{11}\otimes 1$, we have
$$f\left(\mathrm{Mat}(2,\bR[x]/(x^3))\otimes \bR[t]/J^3\right)f\cong \bR[x,t]/(x^3, x^2t, xt^2, t^3).$$
In \cite[(1.2)]{Ringel75}, the following result is attributed to Brenner and Drozd:
\begin{quote}
the commutative algebra $\bR[x,t]/(x^2,xt^2,t^3)$ is wild.
\end{quote}
As $\bR[x,t]/(x^2,xt^2,t^3)$ is a quotient algebra of $\bR[x,t]/(x^3, x^2t, xt^2, t^3)$, $\bR[x,t]/(x^3, x^2t, xt^2, t^3)$ is wild. We have proved that $A$ is wild.
\end{proof}

\subsection{ Representation type of $\KLR(\beta)$}

It remains to consider $\KLR(k\delta)$, for $k\ge3$.

\begin{lemma}[{\cite[Prop.2.3]{EN02}}]
\label{reduction to critical rank}
Let $A$ and $B$ be finite dimensional $\bR$-algebras and suppose that
there exist a constant $C>0$ and functors
$$
F:\;A\text{\rm -mod} \rightarrow B\text{\rm -mod}, \quad
G:\;B\text{\rm -mod} \rightarrow A\text{\rm -mod}
$$
such that, for any $A$-module $M$,
\begin{itemize}
\item[(1)]
$M$ is a direct summand of $GF(M)$ as an $A$-module,
\item[(2)]
$\dim F(M)\le C\dim M$.
\end{itemize}
If $A$ is wild then so is $B$.
\end{lemma}

\begin{prop}\label{k=3}
$\KLR(k\delta)$, for $k\ge2$, are wild.
\end{prop}
\begin{proof}
Suppose that $\ell=1$. Then the numbers
$$
\langle h_0, \ \Lambda_0-k\delta\rangle=1, \quad
\langle h_1, \ \Lambda_0-k\delta-\alpha_0\rangle=1, \quad
\langle h_0, \ \Lambda_0-k\delta-\alpha_0-\alpha_1\rangle=3,
$$
are all positive. Thus, the functors
\begin{align*}
F_0:& \ \KLR(k\delta)\text{\rm -mod}\rightarrow
\KLR(k\delta+\alpha_0)\text{\rm -mod},\\
F_1:& \ \KLR(k\delta+\alpha_0)\text{\rm -mod}\rightarrow \KLR(k\delta+\alpha_0+\alpha_1)\text{\rm -mod},\\
F_0:& \ \KLR(k\delta+\alpha_0+\alpha_1)\text{\rm -mod}\rightarrow \KLR((k+1)\delta)\text{\rm -mod}
\end{align*}
satisfy the assumptions of Lemma \ref{reduction to critical rank}, by Theorem \ref{Thm: categorification}. As $\KLR(2\delta)$ is wild by Proposition \ref{k=2}, so are  $\KLR(k\delta)$, for $k\ge3$.
Suppose that $\ell\ge2$. Then
\begin{multline*}
\langle h_0, \ \Lambda_0-k\delta\rangle, \
\ldots, \ \langle h_\ell, \ \Lambda_0-k\delta-\alpha_0-\cdots-\alpha_{\ell-1}\rangle, \ \ldots \\
\ldots, \ \langle h_1, \ \Lambda_0-(k+1)\delta+\alpha_0+\alpha_1\rangle,, \
\langle h_0, \ \Lambda_0-(k+1)\delta+\alpha_0\rangle,
\end{multline*}
are all positive, and the same argument as $\ell=1$ case proves the result.
\end{proof}

We recall that  $V(\Lambda_0)_\mu\ne0$ if and only if $\mu=\kappa-k \delta $, for some $\kappa\in \weyl\Lambda_0$ and $k\in \Z_{\ge 0}$, and the pair
$(\kappa, k)\in \weyl\Lambda_0\times\Z_{\ge0}$ is uniquely determined by $\mu$. The following theorem is the Erdmann-Nakano type theorem for type $A^{(2)}_{2\ell}$.

\begin{thm} \label{Thm: repn type of R(beta)}
Let $\kappa \in \weyl\Lambda_0$ and $k \in \Z_{\ge 0}$. The finite quiver Hecke algebra
$\KLR(\Lambda_0 - \kappa + k \delta )$ of type $A^{(2)}_{2\ell}$, for $\ell\ge1$, is
\begin{itemize}
\item[(1)]
simple if $k=0$,
\item[(2)]
of finite representation type but not semisimple if $k=1$,
\item[(3)]
of wild representation type if $k\ge2$,
\item[(4)]
and tame representation type does not occur.
\end{itemize}
\end{thm}
\begin{proof}
As the representation type of $\KLR(\Lambda_0 - \kappa + k \delta )$ is the same as  that of $\KLR(k \delta )$ by Corollary \ref{Cor: repn type}, the result follows from Theorem \ref{k=1} and Proposition \ref{k=3}.
\end{proof}
Note that if $k=1$ then the stable Auslander-Reiten quiver is $\mathbb{Z} A_{2\ell}/\langle \tau^\ell \rangle$, where $\tau$ is the Auslander-Reiten translate.

\vskip 1em

\section{Appendix: generalized cellular structure}

\subsection{Generalized cellular algebras}

We first recall K\"onig and Xi's notion of \emph{affine cellular algebra} \cite{KX12}. Let $\bR$ be a field.
We note that they also require that $B$ is commutative in the definition below.

\begin{defn}
Let $(A,\sigma_A)$ and $(B,\sigma_B)$ be $\bR$-algebras with anti-involution.
A two-sided ideal $J\subseteq A$ is called an \emph{affine cell ideal} if
$\sigma_A(J)=J$ and there exist an $(A,B)$-bimodule $C$ and
an $(A,A)$-bimodule isomporphism
$$
\alpha: J\simeq C\otimes_B C^{\rm op},
$$
where $C^{\rm op}$ is $C$ equipped with right $A$-module structure given by
$xa=\sigma_A(a)x$, for $x\in C$ and $a\in A$, such that
\begin{itemize}
\item[(a)]
$C$ is free as a $B$-module.
\item[(b)]
Let $\tau_A:C\otimes_B C^{\rm op}\rightarrow C\otimes_B C^{\rm op}$ be the flip $x\otimes y\mapsto y\otimes x$. Then
$\alpha\circ\sigma_A=\tau_A\circ\alpha$.
\end{itemize}
\end{defn}

Recently, Masaki Mori has modified the definition and introduced more transparent general setup.

\begin{defn}
Let $A$ and $B$ be $\bR$-algebras. We call a pair of $(A,B)$-bimodule $M$ and
$(B,A)$-bimodule $N$ \emph{weak Morita pair} if we have $(A,A)$ and $(B,B)$-bimodule homomorphisms
$$
\varphi_A: M\otimes_B N\longrightarrow A\quad\text{and}\quad \varphi_B: N\otimes_A M\longrightarrow B
$$
such that the following diagrams commute:
$$
\xymatrix{
 M\otimes_{B} N\otimes_{A} M \ar[d]_{ \mathrm{id} \otimes \varphi_B}  \ar[rr]^{\varphi_{A} \otimes \mathrm{id}} & &  A\otimes_{A} M \ar[d]^{\wr} \\
 M \otimes_{B} B  \ar[rr]^{\sim} &     &  M,
}    \quad
\xymatrix{
 N\otimes_{A} M\otimes_{B} N \ar[d]_{  \mathrm{id}  \otimes \varphi_A }  \ar[rr]^{\varphi_{B} \otimes \mathrm{id}} & &  B\otimes_{B} N \ar[d]^{\wr} \\
 N \otimes_{A} A  \ar[rr]^{\sim} &     &  N .
}
$$
\end{defn}

Suppose that $(M,N)$ is a weak Morita pair. We denote $J_A=\im(\varphi_A)$ and $J_B= \im(\varphi_B)$.
Mori has proved the following lemma by showing that any element in $M\otimes_B V\setminus \Ker(\Phi)$ generates
$M\otimes_B V$: the proof resembles the proof of \cite[Prop.2.11]{Mathas99}.

\begin{lemma}
Let $V$ be an irreducible $B$-module with $J_BV\neq0$. We define an $A$-module homomorphism
$$
\Phi: M\otimes_B V \rightarrow \Hom_B(N,V)
$$
by $\Phi(m\otimes v): n\mapsto \varphi_B(n\otimes m)v$. Then, $\Ker(\Phi)$ is the unique maximal proper
$A$-submodule of $M\otimes_B V$.
\end{lemma}

In particular, $\im(\Phi)\simeq M\otimes_B V/ \Ker(\Phi)$ is an irreducible $A$-module, which we denote by $D_M^NV$.
By the commutativity constraint for the Morita pair and $J_BV=V$, we have $J_AD_M^NV\neq0$. Thus, we have a map
$V\mapsto D_M^NV$ from the set of isomorphism classes of irreducible $B$-modules $V$ with $J_BV\neq0$ to the set of
isomorphism classes of irreducible $A$-modules $W$ with $J_AW\neq0$. By interchanging the role of $M$ and $N$,
$W\mapsto D_N^M W$ gives a map in the opposite direction.

Next lemma by Mori implies that $D_M^NV= \Soc( \Hom_B(N,V))$. For the proof, he finds $b\in B$ and $n\in N$ which
satisfy $\Phi(m\otimes v)=\varphi_A(mb\otimes n)f$ for any $m\otimes v \in M\otimes_B V$.

\begin{lemma}
Let $V$ and $\Phi$ be as above.
For any nonzero element $f\in \Hom_B(N,V)$, we have $\im(\Phi)\subseteq Af$.
\end{lemma}

A main corollary of the above two lemmas is the following result by Mori. Mori informed us M\"uller's work \cite{M74}, but
Mori's argument is transparent and enough for our purposes.

\begin{thm} \label{Thm: Mori}
Let $(M,N)$ be a weak Morita pair. Then, $V\mapsto D_M^NV$ induces a bijective map between the set of
isomorphism classes of irreducible $B$-modules $V$ with $J_BV\neq0$ and that of
isomorphism classes of irreducible $A$-modules $W$ with $J_AW\neq0$.
\end{thm}
\begin{proof}
It suffices to show that $W\simeq D_M^NV$ if and only if $V\simeq D_N^MW$. Suppose that
$W\simeq D_M^NV$. Then
$$
0\neq \Hom_A(M\otimes_B V, W)\simeq \Hom_B(V, \Hom_A(M,W))
$$
implies that $V$ appears in $\Soc(\Hom_A(M,W))\simeq D_N^MW$. If $V\simeq D_N^MW$, then
$$
0\neq \Hom_B(N\otimes_A W, V)\simeq \Hom_A(W, \Hom_B(N,V))
$$
implies that $W$ appears in $\Soc(\Hom_B(N,V))\simeq D_M^NV$.
\end{proof}

Following those ideas by K\"onig-Xi and Mori, we shall define as follows.

\begin{defn}
Let $(A,\sigma_A)$ be an $\bR$-algebra with anti-involution. We call $(A,\sigma_A)$ \emph{generalized cellular} if
there exists a sequence of $\sigma_A$-stable two-sided ideals
$$
A=J_0\supseteq J_1\supseteq \cdots\supseteq J_r=0
$$
and a collection of $\bR$-algebras with anti-involution $(B_i,\sigma_{B_i})$ and $(A/J_i,B_i)$-bimodules $C_i$, for $1\leq i\leq r$,
such that $(C_i,C_i^{\rm op})$ is a weak Morita pair with $J_{i-1}/J_i= \im(\varphi_{A/J_i})$, for $1\leq i\leq r$.
\end{defn}

Let us take all $B_i$ to be $\bR$ and $\sigma_\bR$ to be the identity map. Then, in each step, we have
$$
\varphi_{A/J_i}:\;C_i\otimes_\bR C_i^{\rm op}\rightarrow A/J_i, \quad\text{and}\quad C_i^{\rm op}\otimes_A C_i\rightarrow \bR.
$$
We denote the second map by $x\otimes y\mapsto \langle x, y\rangle$. As $(C_i, C_i^{\rm op})$ is a weak Morita pair,
we have
$$
\varphi_{A/J_i}(x\otimes y)z=\langle y, z\rangle x,\quad \langle z, y\rangle x=\sigma_A(\varphi_{A/J_i}(y\otimes x))z.
$$
If we further assume that $\varphi_{A/J_i}$ commutes with anti-involutions, namely, if we assume
$$
\varphi_{A/J_i}(y\otimes x)=\sigma_A(\varphi_{A/J_i}(x\otimes y)),
$$
then
$$
\langle z, y\rangle x=\sigma_A(\varphi_{A/J_i}(y\otimes x))z=\varphi_{A/J_i}(x\otimes y)z=\langle y, z\rangle x,
$$
showing that the bilinear form is symmetric.\footnote{If $\varphi_{A/J_i}$ anti-commutes with anti-involutions,
then the bilinear form is skew-symmetric.} If we put one more assumption
that $\varphi_{A/J_i}$ are monomorphims, then we have
$$
C_i\otimes_R C_i^{\rm op}\simeq J_{i-1}/J_i
$$
and we reach the classical definition of cellular algebras. In fact, the commutativity constraint is known
to hold in cellular algebras \cite[Prop.2.9]{Mathas99}.
Hence, cellular algebras are generalized cellular algebras, and the reader can see that
Theorem \ref{Thm: Mori} is vast generalization of Graham-Lehrer's result on classification of irreducible modules.

\vskip 1em

\subsection{Structure of $\KLR(\beta)$}
Any algebra has the trivial generalized cellular structure:
we choose the cell bimodule  $C$ to be the algebra itself. Hence, our aim in this subsection is to propose a reasonable set of cell modules for $\KLR(\beta)$, for
$\beta \in \rlQ^+$ with $|\beta|=n$.

Set $A=\KLR(\beta)$.
Recall the dominance order on the set of shifted Young diagrams and the canonical tableaux $T^\lambda$ from $\eqref{Eq: order on SYD}$.
We define idempotents $\{ e_\lambda \mid \lambda\vdash n\}$ of $A$ by
$$ e_\lambda = e(  \res(T^\lambda) ). $$

We fix a linear extension of the dominance order, e.g. the lexicographic order, and
define $J^{<\lambda}, B^\lambda, C(\lambda)$, for $\lambda\vdash n$, as follows.
$$J^{<\lambda}=\sum_{\mu < \lambda} Ae_\mu A, \quad
B^\lambda=e_\lambda Ae_\lambda/e_\lambda J^{<\lambda}e_\lambda,  \quad\text{and}\quad
 C(\lambda)=Ae_\lambda/J^{<\lambda}e_\lambda.$$
Then $C(\lambda)$ is an $(A/J^{<\lambda}, B^\lambda)$-bimodule in the natural way.

Let $\sigma_A$ be the anti-involution of $A$ which is the identity on the set of the  generators. Then, we have $C(\lambda)^{\rm op}=e_\lambda A/e_\lambda J^{<\lambda}.$
We define
\begin{align*}
\varphi_{A/J^{<\lambda}} \ &: \ C(\lambda)\otimes_{B^\lambda} C(\lambda)^{\rm op} \longrightarrow A/J^{<\lambda}, \\
\varphi_{B^\lambda} \ &: \ C(\lambda)^{\rm op}\otimes_A C(\lambda) \longrightarrow B^{\lambda},
\end{align*}
by the maps induced by the multiplication on $A$. Then, we have
$$\im(\varphi_{A/J^{<\lambda}})=Ae_\lambda A+J^{<\lambda}/J^{<\lambda}.$$
We sort the two-sided ideals $J^{<\lambda}$ as
$$
A=J_0\supseteq J_1\supseteq \cdots\supseteq J_r=0.
$$
We delete repetition whenever $J_i=J_{i+1}$ occurs.

\begin{lemma} The pair $(C(\lambda), C(\lambda)^{\rm op})$ is a weak Morita pair, and $\KLR(\beta)$ is generalized cellular with respect to the collection of those weak Morita pairs.
\end{lemma}
\begin{proof}
The commutativity constraints obviously hold, because they are given by the product map in $\KLR(\beta)$.
\end{proof}

The generalized cellular structure we propose in the above looks natural, but we know little about the structure of cell modules.
It might be an interesting question to determine the set of irreducible $B^\lambda$-modules $V$ with $\im(\varphi_{B^\lambda})V\ne 0$.


\bibliographystyle{amsplain}


\end{document}